\documentclass[12pt,reqno]{amsart}
\usepackage[margin=1in]{geometry}
\usepackage{amsmath,amssymb,amsthm,graphicx,amsxtra, setspace}
\usepackage[utf8]{inputenc}
\usepackage{mathrsfs}
\usepackage{hyperref}
\usepackage{upgreek}
\usepackage{mathtools}
\usepackage{xcolor}

\usepackage[T1]{fontenc}  % Enables proper accented character output
\usepackage[utf8]{inputenc}
\usepackage[mathcal]{euscript}
\allowdisplaybreaks

\usepackage{tikz}
\usetikzlibrary{decorations.pathmorphing,arrows.meta}

\usetikzlibrary{arrows}

\usepackage[pagewise]{lineno}

\DeclareMathAlphabet{\mathpzc}{OT1}{pzc}{m}{it}

\usepackage[cyr]{aeguill}

\colorlet{darkblue}{blue!50!black}

\hypersetup{
	colorlinks,%
	citecolor=blue,%
	filecolor=red,%
	linkcolor=red,%
	urlcolor=blue,%
	pdfnewwindow=true,%
	pdfstartview={FitH}
}

\newtheorem{theorem}{Theorem}[section]
\newtheorem{lemma}[theorem]{Lemma}
\newtheorem{proposition}[theorem]{Proposition}

\newtheorem{definition}[theorem]{Definition}
\newtheorem{problem}[theorem]{Problem}
\newtheorem{example}[theorem]{Example}
\newtheorem{remark}[theorem]{Remark}

\newtheorem{hypothesis}[theorem]{Hypothesis}

\allowdisplaybreaks

\let\originalleft\left
\let\originalright\right
\renewcommand{\left}{\mathopen{}\mathclose\bgroup\originalleft}
\renewcommand{\right}{\aftergroup\egroup\originalright}

\renewcommand{\d}{\/\mathrm{d}\/}

\def\w{\textbf{W}^{\varepsilon}_{{\theta}^{\varepsilon}}}

\def\L{\mathbb{L}}
\def\A{\mathscr{A}}
\def\U{\mathbb{U}}

\def\C{\mathscr{C}}
\def\f{\boldsymbol{f}}

\def\B{\mathscr{B}}

\def\y{\boldsymbol{u}}
\def\u{\boldsymbol{y}}
\def\z{\boldsymbol{v}}
\def\v{\boldsymbol{z}}

\def\Y{\mathbb{Y}}

\def\X{\mathbb{X}}
\def\x{\boldsymbol{x}}

\def\h{\boldsymbol{h}}
\def\V{\mathbb{v}}
\def\w{\boldsymbol{w}}

\def\N{\mathbb{N}}

\def\V{\mathbb{V}}
\def\wi{\widetilde}

\def\H{\mathbb{H}}

\newcommand{\R}{\mathbb{R}}
\newcommand{\Div}{\mathop{\mathrm{Div}}}
\renewcommand{\d}{\/\mathrm{d}\/}

\newcommand{\Addresses}{{% additional braces for segregating \footnotesize
		\footnote{
			\noindent \textsuperscript{1,2}Department of Mathematics, Indian Institute of Technology Roorkee-IIT Roorkee,
			Haridwar Highway, Roorkee, Uttarakhand 247667, INDIA.\par\nopagebreak
			\noindent  \textit{e-mail:} \texttt{Manil T. Mohan: maniltmohan@ma.iitr.ac.in, maniltmohan@gmail.com.}
			
			\textit{e-mail:} \texttt{Jyoti Jindal: jyoti@ma.iitr.ac.in.}
						
			\noindent \textsuperscript{*}Corresponding author.
			
			\textit{Key words:} Bingham fluids, Convective Brinkman-Forchheimer extended Darcy equations, Variational-hemivariational inequalities, Pseudomonotone operators.
			
			Mathematics Subject Classification (2020): Primary 35R70, 47J20, 35Q35; Secondary 35M86, 76D03.
			
}}}

\begin{document}
	%	\linenumbers
	
\title[A Quasi-Variational--Hemivariational inequality]{A Quasi-Variational--Hemivariational Inequality for the Convective Brinkman--Forchheimer Extended Darcy Equations with Bingham Fluids	\Addresses}
	\author[J. Jindal and M. T. Mohan]
	{Jyoti Jindal\textsuperscript{1} and Manil T. Mohan\textsuperscript{2*}}
	\maketitle
\begin{abstract}
This paper is devoted to the analysis of a quasi-variational--hemivariational inequality associated with the convective Brinkman--Forchheimer extended Darcy (CBFeD) equations for Bingham fluids in both two and three spatial dimensions. The considered model describes incompressible fluid flow through porous media while incorporating convection, nonlinear damping effects, and Forchheimer-type resistance, together with non-smooth and non-convex slip boundary conditions. We first derive an appropriate weak formulation of the problem, which leads naturally to a Bingham-type quasi-variational--hemivariational inequality with a velocity-dependent constraint set. By employing the Kakutani--Ky Fan fixed point theorem, we establish the existence of weak solutions for the resulting multivalued quasi-variational inequality formulation of the CBFeD system. Furthermore, we prove that every weak solution of the associated quasi-variational inequality is also a solution of the corresponding quasi-variational--hemivariational inequality. The analysis presented in this work provides a rigorous mathematical framework for CBFeD models with Bingham fluids under non-monotone and non-smooth boundary interactions.
\end{abstract}

\section{Introduction}\setcounter{equation}{0}

\subsection{The model}\label{subsec1}
 The two- and three-dimensional convective Brinkman--Forchheimer extended Darcy (CBFeD) equations for Bingham fluids are  mathematical models used to describe the motion of \emph{yield-stress non-Newtonian fluids} through \emph{porous media} when several physical mechanisms act simultaneously. Classical Darcy's law is valid only for slow Newtonian flows and cannot accurately capture viscous shear effects, inertial resistance, or yield-stress behavior. To overcome these limitations, the CBFeD model combines the \emph{Darcy term}, representing linear resistance due to the porous matrix; the \emph{Brinkman term}, accounting for viscous diffusion and boundary-layer effects; the \emph{Forchheimer term}, modeling nonlinear inertial drag at moderate or high velocities; and the \emph{convective term}, which incorporates momentum transport through nonlinear advection. For Bingham fluids, an additional constitutive relation introduces a \emph{yield stress}, meaning that the material behaves like a rigid solid until the applied stress exceeds a critical threshold. 
  
 Let $\Omega$ be a bounded domain with Lipschitz boundary $\Gamma$ in $\mathbb{R}^d$ for $d\in \{2,3\}$, and let $\mathbb{M}^d$ denote the space of symmetric $d \times d$ matrices. The boundary $\Gamma$  is decomposed into two disjoint measurable subsets, $\Gamma_0$ and $\Gamma_1$, such that their $(d-1)$-dimensional Hausdorff measures satisfy $|\Gamma_0|>0$ and  $|\Gamma_1|>0.$ The governing momentum equation is  written as
\begin{align}\label{Main eq}
	-\Div \mathbb{S} + \Div (\u \otimes \u)+ \alpha \u +\beta |\u|^{r-1}\u+\kappa |\u|^{q-1}\u+ \nabla p= \f &  \ \ \text{ in }\  \  \Omega,
\end{align}
\begin{equation}\label{Ten-1}
	\left\{
	\begin{aligned}
		\mathbb{S} =& \mathbb{T}(\mathbb{D}\u)
		+ g \frac{\mathbb{D}\u}{\|\mathbb{D}\u\|_{\mathbb{M}^d}} \ 
	\ 	\text{ if } \ \mathbb{D}\u \neq 0   \ \text{ in }\    \Omega, \\
		\|\mathbb{S}\|_{\mathbb{M}^d} &\le g \ \text{ if } \ \mathbb{D}\u = 0,
	\end{aligned}
	\right.
\end{equation}
\begin{align}\label{4.3}
	\nabla\cdot\u = 0 \ \ \text{ in }\  \ \Omega.
\end{align}
 These equations are highly relevant in engineering and geophysical applications such as the flow of drilling mud in petroleum reservoirs, fresh concrete through reinforced structures, lava transport in fractured rocks, slurry filtration, and viscoplastic biofluid transport in tissues (\cite{IRV+BV-1988}). The CBFeD framework is important because it captures realistic interactions between porous resistance, nonlinear inertia, viscous diffusion, and yield-stress rheology, thereby providing a more accurate description of complex non-Newtonian porous flows than classical Darcy or Navier--Stokes models alone (\cite{MSD1,MSSD1}).

In the model \eqref{Main eq}, the velocity field is denoted by $\u : \Omega \to \R^d$, the pressure is denoted by $p : \Omega \to \R$, and
$\f : \Omega  \to \R^d$ stands for the external force. The positive constants $\alpha$ and $\beta$ corresponds to the Darcy and Forchheimer damping effects, which account for resistance due to permeability and
porosity, respectively. Furthermore, the absorption or damping term with \emph{absorption exponent} $r\geq1$ is the nonlinear term $\beta|\u|^{r-1}\u$ (cf. \cite{SNAHB}).
Specifically, for $\alpha=\kappa=0$, the exponent $r=3$ is called the \emph{critical exponent} due to scaling invariance. Indeed, under the scaling $\u_\lambda(x)=\lambda \u(\lambda x),$ the Brinkman term and the convective term scale like $\lambda^3$, while the nonlinear damping term $|\u|^{r-1}\u $ scales like $\lambda^r$. Hence, the equation remains invariant precisely when $r=3$. The cases $r>3$ and $r<3$ are referred to as the \emph{supercritical} and \emph{subcritical} cases, respectively.

As observed in \cite{MTT}, $\kappa |\u|^{q-1}\u$ is the nonlinear term  which appears in \eqref{Main eq} and models a pumping mechanism. Under the assumption $\kappa<0$, adopted throughout this paper, this term acts against the damping effect present in the system. The exponent $1 \leq q <r$ determines the intensity of the pumping mechanism. The governing system includes several classical fluid models as limiting cases. In particular, choosing $\alpha=\beta=\kappa=0$ eliminates the additional nonlinear terms, thereby yielding the classical Navier--Stokes equations for Bingham fluids. Moreover, when $\alpha,\beta>0$ and $\kappa=0$, the system corresponds to the damped Navier--Stokes equations with Bingham-type rheology.

We follow the work \cite{MSD1} for the notations and boundary conditions for the system \eqref{Main eq}-\eqref{4.3}. In particular, the constitutive relation involves the symmetric part of the velocity gradient, which measures the local rate of deformation of the fluid. Accordingly, in \eqref{Ten-1}, the rate of deformation tensor is defined by
\begin{align*}
	\mathbb{D}\u = \frac{1}{2}\big(\nabla \u + (\nabla \u)^{\top}\big),
\end{align*}
which characterizes the symmetric part of the velocity gradient, and represents the strain rate
in the fluid. The associated \emph{total stress tensor} is determined by
$$
\boldsymbol{\sigma}(\u,p) = -p \mathbb{I} + \mathbb{S}(\mathbb{D}\u)   \ \text{ in }\  \Omega,
$$
where $\mathbb{I}$ is the identity matrix and $\mathbb{S}$ is the \emph{extra stress tensor}. Also, $\Div$ is the divergence operator for tensors, that is,
 $$
 \Div(\mathbb{S})_i := \sum_{j=1}^d \frac{\partial}{\partial y_j} \mathbb{S}_{ij}.
 $$
For given two vectors $\boldsymbol{a},\boldsymbol{b} \in \R^d$, the second-order tensor $\boldsymbol{a}\otimes\boldsymbol{b}$ is obtained by taking their tensor product, given by
$$\boldsymbol{a} \otimes \boldsymbol{b} = (a_i b_j)_{1 \leq i,j \leq d}.$$
Additionally, the Bingham constitutive law in \eqref{Ten-1} is a nonlinear relation between the symmetric part of velocity gradient $\mathbb{D} \u$ and extra stress tensor $\mathbb{S}$, where $g$ represents the plasticity threshold (yield limit) and describes the flow behaviour of a fluid with a yield stress. 
%Moreover, the function $\mathbb{T}: \Omega \times \mathbb{M}^d \to \mathbb{M}^d$ has the following form:
%\begin{align}\label{4.11}
%	\mathbb{T} (x, \mathbb{D}) = \delta(\|\mathbb{D}\|_{\mathbb{M}^d}) \mathbb{D}\ \text{ for }\  \mathbb{D} \in \mathbb{M}^d, \text{ a.e. } x \in \Omega,\end{align}
%where the specified viscosity function is $\delta:[0,\infty)\to \mathbb{R}$. If $\delta(r)=\delta_0$ for all $r\geq 0$, where $\delta_0>0$ is a prescribed viscosity constant, then \eqref{4.11} simplifies to
%$$\mathbb{T}(x,\mathbb{D})=\delta_0\mathbb{D}.$$
%This corresponds to the classical linear constitutive relation for a Newtonian fluid and satisfies the assumptions of Hypothesis \ref{hyp-T}. Furthermore, the constitutive law \eqref{Ten-1} reduces to the Bingham fluid model when $\delta(r)=\delta_0$ for all $r\geq 0$, and the system reduces to the CBFeD model in the case $g=0$. 
Finally, the divergence-free condition in \eqref{4.3} expresses the incompressibility of the fluid and ensures conservation of mass.

To complete the mathematical formulation of the CBFeD system given in \eqref{Main eq}-\eqref{4.3}, it is necessary to prescribe appropriate boundary conditions on the boundary of the flow domain. These conditions describe the interaction between the fluid and the surrounding boundary, accounting for both adherence and slip phenomena on different parts of the boundary. Accordingly, the system \eqref{Main eq}--\eqref{4.3} is supplemented with the following boundary conditions:
\begin{align}
	\u &= \boldsymbol{0} \quad  \ \text{ on }\  \Gamma_0, \label{bc1}\\
	y_\nu &= 0, \quad 
	-\boldsymbol{\tau}_\tau(\u) \in k(\u_\tau)\,\partial j_\tau(\u_\tau)
	\ \text{ on }\  \Gamma_1 \label{bc2},
\end{align}
where $\boldsymbol{\nu}= (\nu_1,\ldots,\nu_d)$ denotes the outward unit normal vector and the \emph{traction vector} on the boundary $\Gamma$ is given by
$$\boldsymbol{\tau}(\u,p) = \boldsymbol{\sigma}(\u,p)\boldsymbol{\nu}.$$  In addition, the normal and tangential components of the velocity and traction are defined by
$$
y_\nu = \u \cdot \boldsymbol{\nu}, \qquad
\u_{\tau} = \u - y_\nu\boldsymbol{\nu},
$$
and
$$
\tau_{\nu}(\u,p) = \boldsymbol{\tau}(\u,p)\cdot\boldsymbol{\nu}, \qquad
\boldsymbol{\tau}_{\tau}(\u) = \boldsymbol{\tau}(\u,p)
- \tau_{\nu}(\u,p)\boldsymbol{\nu},
$$
respectively. In addition, the normal and tangential components of the extra stress tensor are defined as
\begin{align}\label{1.6}
	\mathbb{S}_{\nu}(\u)=\tau_{\nu}(\u,p)+p,  \qquad \mathbb{S}_{\tau}(\u)=\boldsymbol{\tau}_{\tau}(\u) \ \text{ on }\  \Gamma.
\end{align}
These decompositions represent the projections of the velocity and traction vectors onto the normal and tangential directions of the boundary, respectively. In the boundary condition \eqref{bc2}, $\partial j_\tau$ refers to the Clarke subdifferential of the function $j_\tau(x,\cdot)$, where the function $j_\tau : \Gamma_1 \times \mathbb{R}^d \to\mathbb{R}$ is known as the superpotential and $j_\tau(x, \u_\tau)$ is denoted by $j_\tau(\u_\tau)$. Also, $k: \Gamma_1 \times \mathbb{R}^d \to\mathbb{R}$ is the functional satisfying the assumptions given in Hypothesis \ref{hyp-k}. The boundary condition \eqref{bc1} represents the no-slip condition, indicating that the fluid adheres to the rigid wall. The condition \eqref{bc2} describes a non-monotone friction-type slip condition on the boundary and is frequently used in modeling complex flow phenomena, such as fluid flow through drains or canals whose bottoms are covered with layers of mud and pebbles, or more generally, flow over rough porous surfaces \cite{HFU}.

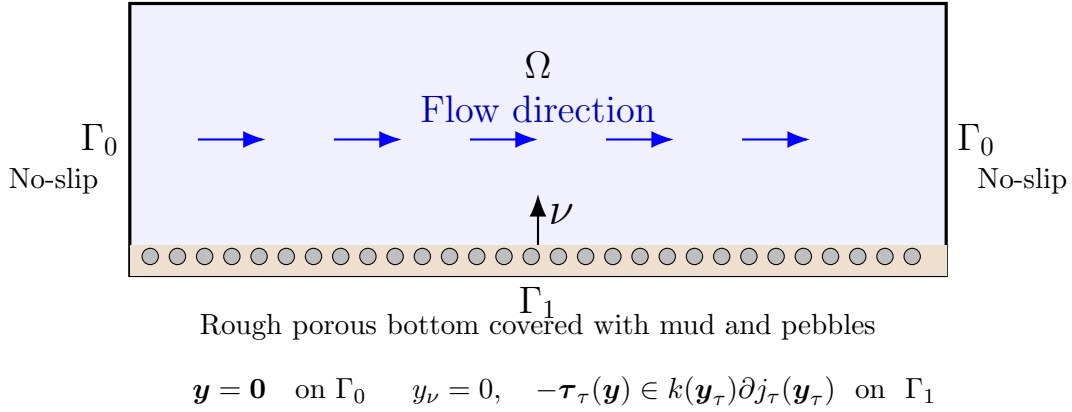
\begin{figure}[ht]
	\centering
	
	\begin{tikzpicture}[scale=0.9]
		
		%---------------------------------------------------------
		% Main bounded domain
		%---------------------------------------------------------
		
		% Fluid region
		\fill[blue!6] (0,0) rectangle (12,4);
		
		% Boundary of domain
		\draw[very thick] (0,0) rectangle (12,4);
		
		%---------------------------------------------------------
		% Rough porous bottom
		%---------------------------------------------------------
		
		\fill[brown!25] (0,0) rectangle (12,0.45);
		
		% Pebbles
		\foreach \x in {0.3,0.7,1.1,1.5,1.9,2.3,2.7,3.1,3.5,3.9,4.3,4.7,5.1,5.5,5.9,6.3,6.7,7.1,7.5,7.9,8.3,8.7,9.1,9.5,9.9,10.3,10.7,11.1,11.5}
		{
			\draw[fill=gray!50] (\x,0.28) circle (0.12);
		}
		
		%---------------------------------------------------------
		% Flow arrows
		%---------------------------------------------------------
		
		\draw[-{Latex[length=3mm]}, thick, blue]
		(1.0,2.0) -- (2.0,2.0);
		
		\draw[-{Latex[length=3mm]}, thick, blue]
		(3.0,2.0) -- (4.0,2.0);
		
		\draw[-{Latex[length=3mm]}, thick, blue]
		(5.0,2.0) -- (6.0,2.0);
		
		\draw[-{Latex[length=3mm]}, thick, blue]
		(7.0,2.0) -- (8.0,2.0);
		
		\draw[-{Latex[length=3mm]}, thick, blue]
		(9.0,2.0) -- (10.0,2.0);
		
		%---------------------------------------------------------
		% Labels
		%---------------------------------------------------------
		
		\node at (6,3.1) {\large $\Omega$};
		
		\node[left] at (0,2)
		{\large $\Gamma_0$};
		
		\node[right] at (12,2)
		{\large $\Gamma_0$};
		
%		\node[above] at (6,4)
%		{\large $\Gamma_0$};
		
		\node[below] at (6,0)
		{\large $\Gamma_1$};
		
		\node[blue!70!black] at (6,2.45)
		{\large Flow direction};
		
		%---------------------------------------------------------
		% Normal vector
		%---------------------------------------------------------
		
		\draw[-{Latex[length=3mm]}, thick]
		(6,0.45) -- (6,1.2);
		
		\node[right] at (6,0.9)
		{\Large $\nu$};
		
		%---------------------------------------------------------
		% Wall labels
		%---------------------------------------------------------
		
		\node[left] at (-0.3,1.4)
		{\small No-slip};
		
		\node[right] at (12.3,1.4)
		{\small No-slip};
		
%		\node[above] at (7.5,4.25)
%		{\small No-slip wall};
		
		%---------------------------------------------------------
		% Bottom text
		%---------------------------------------------------------
		
		\node at (6,-0.75)
		{\small Rough porous bottom covered with mud and pebbles};
		
		%---------------------------------------------------------
		% Equations
		%---------------------------------------------------------
		
		\node at (2.2,-1.7)
		{\small
			$
			\u=\boldsymbol{0}
			\quad \text{on } \Gamma_0
			$
		};
		
		\node at (8.0,-1.7)
		{\small
			$
			y_\nu=0,
			\quad  
			-\boldsymbol{\tau}_\tau(\u)
			\in
			k(\u_\tau)\partial j_\tau(\u_\tau) 	\ \text{ on }\  \Gamma_1 
			$
		};
		
	\end{tikzpicture}
	\caption{Schematic representation of flow through a drain or porous filter with a rough bottom covered by mud and pebbles. The rigid walls $\Gamma_0$ satisfy the no-slip condition, whereas the rough porous boundary $\Gamma_1$ allows frictional slip or leakage phenomena modeled by a non-monotone boundary condition.}
	%	\label{fig:bounded-domain}
	\end{figure}
\subsection{Literature survey} 
Hemivariational inequalities (HVIs) were first introduced by Panagiotopoulos in the
early 1980s \cite{PDP1} to overcome the difficulties arising in the modeling and analysis of engineering problems involving non-smooth and non-monotone interactions between physical quantities. The fundamental aspects of the theory were further developed in several early monographs (see \cite{JHMMPD,ZNPDP,PDP2}). More recent advances in
both mathematical analysis and engineering applications, especially in contact
mechanics, have stimulated extensive research on variational–hemivariational inequalities. Notable developments in this direction can be found in \cite{SMAOMS,  SMAOMS1, MSMS}. 

Hemivariational inequalities (HVIs) also emerge in various applications within fluid mechanics (see \cite{SM1}). The fluid of Bingham type is one of the most important non-Newtonian fluids, which behaves as a rigid body when the applied stress is below a certain threshold, but flows as a viscous incompressible fluid once this threshold is exceeded. A common example is magma within the earth’s crust, mud, liquid crystal polymers, and toothpaste which remains solid in the middle of the plug while exhibiting fluid-like behaviour near the tube wall when stress is applied. Due to its wide range of applications in physics, engineering, and industrial processes, the Bingham fluid has attracted considerable attention in the literature \cite{AAEHMA, BAMSMS, LBOHEP}. 

The mathematical
formulation of Bingham fluid was first introduced by Eugene Bingham, who investigated the plasticity constitutive equations of Bingham
fluid in \cite{BEC}. Subsequently, in 1965, a variational framework for viscoplastic media was then created by Mosolov and Miasnikov in \cite{MSPP}, which is a node in the Bingham fluid transition from observation to theoretical modeling. Next, Duvaut and Lions used variational inequality to study Bingham fluid in \cite{DGLJ}. More research on Bingham fluid for variational-hemivariational inequalities can be found in \cite{ZFHB, MSD1}.

More recently, Migórski and Dudek in \cite{MSSD1} investigated the existence of solutions to the incompressible Navier--Stokes system with Bingham fluid model under mixed boundary conditions and unilateral constraints depending on the solution. In their work, firstly the authors examined the quasi-variational inequality involving a set-valued map by using classical Lions--Stampacchia theory for variational inequalities and  Kakutani--Ky Fan  fixed point theorem for set-valued mappings. Subsequently, they extended the analysis to the corresponding quasi-variational--hemivariational inequality. Also, in this direction, the authors in \cite{XTTC} used the $P2-P1$ finite element to discretize the mixed hemivariational inequality and derive the error estimates.
They also examined the existence and uniqueness of solutions based on the minimization argument. For more information regarding the analysis of quasi-variational--hemivariational inequalities for the incompressible Navier--Stokes with applications to Bingham-type fluids, one can see \cite{YLSMSD, MSYCJS, MSSD2, CWYX} and references therein.

In addition, damping effects also play a crucial role in the mathematical modeling of fluid flows with resistance, as they represent important physical phenomena such as drag forces, viscous dissipation, and other energy-loss mechanisms \cite{SGKKMTM, SGMTM, Hajduk2017, JSMT, JSMT1, MTT}. These effects are typically incorporated into the Stokes and Navier--Stokes equations through damping terms, which enhance model’s ability to capture dissipative behavior in the fluid system. Previous studies have carried out both analytical and numerical investigations of stationary Navier--Stokes hemivariational inequalities arising in the modeling of viscous incompressible fluids with nonlinear damping and pumping effects (see \cite{WAMTM, MTMW}). 

\subsection{Contributions of our work}
To the best of our knowledge, the theoretical study of quasi-variational--hemivariational inequalities associated with the Navier--Stokes equations for viscous incompressible Bingham fluids in the presence of both damping and pumping effects has not yet been addressed in the literature. The present work aims to bridge this gap by formulating and analyzing a quasi-variational--hemivariational inequality corresponding to the convective Brinkman--Forchheimer equations with damping (CBFeD) under an implicit obstacle-type constraint. This framework provides a deeper mathematical understanding of such complex non-Newtonian fluid models. The main novelties and challenges of the present work can be summarized as follows:
\begin{itemize}
\item We investigate a quasi-variational--hemivariational inequality associated with the Navier--Stokes equations incorporating both damping and pumping effects in Bingham fluids. In this way, we substantially extend the results obtained in \cite{MSSD1} to a considerably broader class of fluid flow models.
	
\item In contrast to many existing works (see, e.g., \cite{SDSM, SMAOMS1}), our analysis does not rely on the relaxed monotonicity assumption for the Clarke subgradient. This allows us to treat a wider class of nonsmooth and nonmonotone interactions.
	
\item The presence of multivalued and nonmonotone frictional boundary conditions represented by $k\partial j_\tau$ gives rise to additional analytical difficulties. In general, this nonlinearity cannot be characterized solely by a hemivariational inequality, since there does not exist a potential $\Psi$ such that $	\partial \Psi = k\partial j_\tau.$ Consequently, the model exhibits a genuinely coupled structure: the friction law contributes a \emph{hemivariational term} on the boundary, whereas the Bingham constitutive law induces a \emph{variational term} in the interior of the domain.
	
	\item One of the major technical challenges is establishing the boundedness of the set-valued map $\Lambda$ in the three-dimensional case (see Theorem \ref{theorem-ex}). When $d=3$ and $1\le r\le 5$, the Sobolev embedding	$\V\hookrightarrow \L^{r+1}$ enables the $\L^{r+1}$-norm to be controlled directly by the $\V$-norm, allowing the analysis to be carried out entirely in $\V$. However, for $d=3$ and $r\in(5,\infty)$, this embedding no longer holds, and the natural framework becomes	$\V\cap \L^{r+1}.$
	This lack of embedding introduces substantial additional difficulties and necessitates a much more delicate treatment of the nonlinear terms and the associated a priori estimates.
	
\end{itemize}

\subsection{Structure of the Paper}
The rest of the paper is organized as follows. Section \ref{sec2} introduces the essential mathematical preliminaries, including the fundamental concepts of generalized directional derivative in the sense of Clarke and subdifferential calculus together with their basic properties, and establishes the functional framework associated with the model described in Subsection \ref{subsec1}. We then carry out a detailed analysis of the bilinear, and nonlinear operators arising in the formulation, with particular emphasis on their main properties.

Section \ref{sec3} is devoted to the derivation of the variational formulation associated with the system \eqref{Main eq}--\eqref{bc2}, leading to a quasi-variational--hemivariational inequality. In Section \ref{sec4}, we establish the existence of weak solutions for \emph{Problem \ref{Problem 2.17}}. To this end, we first investigate the auxiliary \emph{Problem \ref{prob3.4}} and prove the existence of its solution in Theorem \ref{theorem3.5}. Subsequently, we introduce the solution operator
\begin{align*}
\Pi: \V\cap\L^{r+1} \times \L^2(\Gamma_1) \to \V\cap\L^{r+1},
\end{align*}
defined by
$$\Pi(\v_0,\w_0)=\u,$$
where $\u \in \V\cap\L^{r+1}$ denotes the unique solution of \emph{Problem \ref{prob3.4}} corresponding to $(\v_0,\w_0) \in \V\cap\L^{r+1} \times \L^2(\Gamma_1)$ and complete continuity of the operator $\Pi$ is established in Lemma \ref{lemma-2}. By combining this result with appropriate compactness arguments and the smallness conditions \eqref{small-1}--\eqref{small-2}, we prove the existence of a solution to \emph{Problem \ref{Problem 3.1}} in Theorem \ref{theorem-ex}. Finally, we show that every solution of \emph{Problem \ref{Problem 3.1}} is also a solution of \emph{Problem \ref{Problem 2.17}}, thereby establishing the connection between the auxiliary and the original problems.

%%%%%%%%%%%%%%%%%%%%%%%%%%%%%%%%%%%
%%%%%%%%%%%%%%%%%%%%%%%%%%%%%%%%%%%
%%%%%%%%%%%%%%%%%%%%%%%%%%%%%%%%%%%
\section{Mathematical Formulation}\label{sec2}\setcounter{equation}{0}
The main objective of this section is to introduce the essential mathematical preliminaries needed for the theoretical analysis developed in this work.

\subsection{Preliminaries}
In this work, function spaces are defined over the field of real numbers. Let $\X$ be a Banach space with norm $\|\cdot\|_{\X}$ and its topological dual is $\X'$. The duality pairing between $\X'$ and $\X$ is denoted by ${}_{\X'}\langle\cdot,\cdot\rangle_{\X}$, and $\X_w$ denotes $\X$ endowed with the weak topology. We use the symbol $\multimap$ to denote a set-valued mapping. Furthermore, $\mathscr{L}(\X,\Y)$ is the Banach space of bounded linear operators from $\X$ to a Banach space $\Y$. We now begin by recalling the definition of a locally Lipschitz function. Then, the Clarke generalized directional derivative and generalized gradient (Clarke subdifferential) associated with a locally Lipschitz function are defined.

\begin{definition}[{\cite[Definition 3.20]{SMAOMS}}]
A function $\Upsilon:\X\to\mathbb{R}$ is called \emph{locally Lipschitz} if, for every $\x\in\X$, there exists a neighborhood $U$ of $\x$ and a positive constant $L_U$ such that $$|\Upsilon(\y)-\Upsilon(\z)|\leq L_U\|\y-\z\|_{\X}\ \text{ for all }\ \y,\z\in U.$$ 
\end{definition}

\begin{definition}[{\cite[Definition  5.6.3]{ZdSm1}}]
Assume that $\Upsilon:\X \to\R$ is a locally Lipschitz function. The \emph{generalized directional derivative} of $\Upsilon$ at a point $\x\in\X$ in the direction $\v\in\X$ is denoted by $\Upsilon^0(\x;\v)$ and defined as
	\begin{align*}
		\Upsilon^0(\x;\v)=\lim_{\y\to\x}\sup_{\lambda\downarrow 0}\frac{\Upsilon(\y+\lambda\v)-\Upsilon(\y)}{\lambda}. 
	\end{align*}
	The \emph{subdifferential} of $\Upsilon$ at $\x$ in the sense of Clarke, denoted by $\partial \Upsilon(\x)$, is a subset of the dual space $\X^{\prime}$ defined by
	\begin{align}\label{subgradient}
		\partial \Upsilon(\x)=\left\{\boldsymbol{\zeta}\in\X^{\prime}:\Upsilon^0(\x;\v)\geq {}_{\X^{\prime}}\langle\boldsymbol{\zeta},\v\rangle_{\X}\  \text{ for all }\ \v\in\X \right\}. 
	\end{align}
A locally Lipschitz function $\Upsilon$ is called \emph{regular (in the sense of Clarke)} at $\x \in \X$  if the one-sided directional derivative $\Upsilon'(\x;\v)$ exists for all $\v\in\X$, and $\Upsilon^0(\x;\v) = \Upsilon'(\x;\v)$.
\end{definition}
We subsequently review various notions of monotonicity and pseudomonotonicity in the context of single-valued operators.

\begin{definition}[{\cite[Definition 25.2]{EZ}}]
	A single-valued operator $\mathcal{G} : \X\to\X^{\prime}$  is referred to as
\begin{enumerate}
		\item[(i)] \emph{monotone}, if and only if
$${}_{\X^{\prime}}\langle \mathcal{G}(\u)-\mathcal{G} (\v), \u-\v \rangle_{\X} \geq 0  \ \text{ for all }\  \u,\v \in \X;$$
		\item[(ii)] \emph{strongly monotone}, if and only if
$$	{}_{\X^{\prime}}\langle \mathcal{G}(\u)-\mathcal{G}(\v), \u-\v \rangle_{\X} \geq m_{\mathcal{G}}\|\u-\v\|_{\X}^2
	\ \text{ for all }\  \u,\v \in \X \text{ with } m_{\mathcal{G}}>0.$$
	\end{enumerate}
\end{definition}

\begin{definition}[{\cite[Definition 1]{SMAO}}]\label{def-pseudo}
	A single-valued operator $\mathcal{G} : \X\to\X^{\prime}$  is called \emph{pseudomonotone}, if
	\begin{enumerate}
		\item $\mathcal{G}$ is bounded (that is, it maps bounded subsets of $\X$ into bounded subsets of $\X^{\prime}$);
		\item $\u_n\xrightarrow{w}\u$ in $\X$ and $\limsup\limits_{n\to\infty} {}_{\X^{\prime}}\langle\mathcal{G}(\u_n),\u_n-\u\rangle_{\X}\leq 0$ imply 
		\begin{align*}
			{}_{\X^{\prime}}\langle\mathcal{G}(\u),\u-\v\rangle_{\X}\leq\liminf_{n\to\infty}{}_{\X^{\prime}}\langle\mathcal{G}(\u_n),\u_n-\v\rangle_{\X} \ \text{ for all }\  \v\in\X.
			\end{align*}
	\end{enumerate}
\end{definition}
It is proved in \cite[Remark 2]{SMAO} that an operator $\mathcal{G} : \X\to\X^{\prime}$ is pseudomonotone if and only if it is bounded and $\u_n\xrightarrow{w}\u$ in $\X$ combined with $\limsup\limits_{n\to\infty} {}_{\X^{\prime}}\langle\mathcal{G}(\u_n),\u_n-\u\rangle_{\X}\leq 0$ imply 
\begin{align}\label{eqn-con-pseudo}
	\mathcal{G}(\u_n)\xrightarrow{w}\mathcal{G}(\u)\ \text{ in }\ \X'\ \text{ and }\ \lim_{n\to\infty}{}_{\X^{\prime}}\langle\mathcal{G}(\u_n),\u_n-\u\rangle_{\X}=0. 
\end{align}
The following  result will be used to establish the convergence of the nonlinear term.  The result is an immediate consequence of  reflexivity of $\mathrm{L}^q(\Omega)$ with $1<q <\infty$  and Vitali's theorem. 
\begin{lemma}\label{Lem-Lions}
	Let $\Omega \subset \R^d$ be a bounded open set, and $\varphi_m$, $\varphi$ be functions in $\mathrm{L}^q(\Omega)$ with $1<q <\infty$ for $m\in\N$, such that
	\begin{align*}
		\|\varphi_m\|_{\mathrm{L}^q(\Omega)} \leq C\ \text{ for all }\ m\in\N \; \; \mbox{and} \ \ \varphi_m \to \varphi\;\; \mbox{a.e. in}\ \ \Omega,\  \mbox{as}\ \ m \to \infty.
	\end{align*}
	Then, $\varphi_m \xrightarrow{w} \varphi$ in $\mathrm{L}^q(\Omega)$, as $m\to \infty$.
\end{lemma}

\begin{definition}  A function $\varphi \colon \X \to \R$ is considered to be
	\begin{enumerate} 
		\item[(i)] \emph{lower semicontinuous (l.s.c.)} if  $\x_n \to \x$ in $\X$, the following condition holds:
$$ \varphi(\x) \leq \liminf_{n \to \infty} \varphi(\x_n).$$
\item[(ii)] \emph{weakly lower semicontinuous (weakly l.s.c.)} if
	$\x_n \xrightarrow{w} \x$ in $\X$, the following condition holds:
	$$ 	 \varphi(\x) \leq \liminf_{n \to \infty} \varphi(\x_n).$$
		\end{enumerate}
\end{definition}

Observe that every continuous function is lower semicontinuous, and every weakly lower semicontinuous function is also lower semicontinuous. In general, the converse implications are not valid. Despite this, the following result holds.

\begin{proposition}[{\cite[Corollary 3.9]{HRB}}]\label{prop 2.9}
Let $\varphi \colon \X\to \R$ be a convex function. Then $\varphi$ is lower semicontinuous if and only if it is weakly lower semicontinuous.
\end{proposition}
The subsequent definition of Mosco-convergence can be found in  \cite{UM} or \cite[Section 4.7]{ZdSm1} written as follows:
\begin{definition}\label{def:mosco}
Let  $\Y$ be a normed space. In the \emph{Mosco sense}, a sequence $\{\mathcal{C}_n\}_{n \in \N}$ of closed and convex sets in $\Y$ is said to converge to a closed and convex set $\mathcal{C} \subset \Y$, represented by
$$\mathcal{C}_n \xrightarrow{M} \mathcal{C} \  \text{as } \ n \to \infty,$$
if both of the following conditions are satisfied:
	\begin{itemize}
\item[(M$_1$)] For every $\z_n \in \mathcal{C}_n$ with $\z_n \xrightarrow{w} \z$ in $\Y$, up to a subsequence, we have $\z \in \mathcal{C}$.
\item[(M$_2$)] For every $\z \in \mathcal{C}$, there exists $\z_n \in \mathcal{C}_n$ such that $\z_n \to \z$ in $\Y$.
	\end{itemize}
\end{definition}

To establish the existence of a solution, we employ a fixed point argument for multivalued mappings in reflexive Banach spaces. The following result will be instrumental in our analysis.

\begin{theorem}[{\cite[Theorem 3.5]{RK}}]\label{theorem 2.10}
	Let $\X$ be a reflexive Banach space and $\mathcal{D}\subseteq\X$ be a nonempty, bounded, closed, and convex set. Suppose that $\Xi:\mathcal{D}\multimap\mathcal{D}$ is a set-valued mapping with nonempty, closed, and convex values whose graph is sequentially closed in the weak topology $\X_w\times\X_w$. Then $\Xi$ admits a fixed point.
\end{theorem}

To formulate problem \eqref{Main eq}--\eqref{bc2} in a variational framework, we first introduce the functional setting that will be employed throughout the paper.

\subsection{Functional Setting}
In this subsection, we define the function spaces required for the weak formulation of the problem and for the subsequent analysis. These spaces provide the natural framework for describing the velocity, and associated quantities arising in the model.
\begin{align*}
	\mathscr{M}=\left\{\v\in \mathrm{C}^{\infty}(\overline{\Omega};\mathbb{R}^d):\nabla\cdot\v=0 \ \text{ in }\ \ \Omega,\ \v=\boldsymbol{0} \ \text{ on }\ \ \Gamma_0, \ z_\nu=0 \ \text{ on }\ \ \Gamma_1\right\},
\end{align*}
where $z_\nu=\v\cdot\boldsymbol{\nu}$ and $\v_{\tau} =\v-z_\nu\v$ are the normal and tangential components of the vector field $\v$, respectively. Let $\V$ be the closure of $\mathscr{M}$ with respect to the norm of $\mathrm{H}^1(\Omega;\mathbb{R}^d)$, which can be characterized as:  
  \begin{align*}
  \mathbb{V}=\left\{\v\in \mathrm{H}^{1}(\Omega;\mathbb{R}^d):\nabla\cdot\v=0 \ \text{ in }\ \ \Omega,\ \v=\boldsymbol{0} \ \text{ on }\ \ \Gamma_0, \ z_\nu=0 \ \text{ on }\ \ \Gamma_1\right\},
  \end{align*}
and define  $\H=\mathrm{L}^2(\Omega;\mathbb{R}^d)$. We consider norm on the space $\mathbb{V}$ as
%We consider two norms on the space $\mathbb{V}$. The first is the standard norm
%$\|\v\| = \|\v\|_{\mathrm{H}^{1}(\Omega;\mathbb{R}^d)}$, and the second is defined by
\begin{align}\label{norm-1}
	\|\v\|_{\mathbb{V}} := \|\mathbb{D}\v\|_{\mathrm{L}^{2}(\Omega;\mathbb{M}^d)}\  \text{ for all } \ \v \in \mathbb{V}.
\end{align} 
By using Korn’s inequality (see \cite[Theorem 8]{SDPKSM}), there exists a constant $C_k>0$ such that
\begin{equation}\label{2.3}
	\|\v\|_{\mathrm{H}^{1}(\Omega;\mathbb{R}^d)} \le C_k \|\mathbb{D}\v\|_{\mathrm{L}^{2}(\Omega;\mathbb{M}^d)},
\end{equation}
and consequently, the norm on $\mathbb{V}$ given by \eqref{norm-1}
is equivalent to the $\mathrm{H}^1(\Omega;\mathbb{R}^d)$ norm. Using \cite[Theorem 6.1-8]{PGC-1}, we have the following Poincar\'e inequality:
\begin{align*}
	\|\v\|_{\mathbb{H}}\leq C_p\|\nabla \v\|_{\mathbb{H}}\  \text{ for all } \ \v \in \mathbb{V}.
\end{align*}
 An application of Korn’s inequality given in \eqref{2.3} yields 
\begin{align}\label{2p5}
	\|\v\|_{\mathbb{H}}\leq C_pC_k \|\mathbb{D}\v\|_{\mathrm{L}^{2}(\Omega;\mathbb{M}^d)}\  \text{ for all } \ \v \in \mathbb{V}.
\end{align}
For $p\in(2,\infty)$, we denote by ${\L}^p$, the closures of $\mathscr{M}$ with respect to the   $\mathrm{L}^p(\Omega;\mathbb{R}^d)$ norm. The space ${\L}^p$ can also be  characterized in a similar way as above. Then, we have the following continuous embeddings:
\begin{align*}
\V\cap\L^p\hookrightarrow	\V\hookrightarrow\H\equiv\H^{\prime}\hookrightarrow\V^{\prime}\hookrightarrow \V^{\prime}+\L^{\frac{p}{p-1}},
\end{align*}
with the embedding $\V\hookrightarrow\H$ is compact. We denote the norm in ${\L}^p, $ $p\in(2,\infty)$ by 
\begin{align*}
	\|\u\|_{\L^{p}}=\bigg(\int_{\Omega}|\u(x)|^p\d x\bigg)^{1/p}.
\end{align*}
%In the sequel, the duality pairing between the spaces $\V$ and $\V^{\prime}$, $\L^{p}$ and $\L^{\frac{p}{p-1}}$, and $\V\cap\L^{p}$ and $\V^{\prime}+\L^{\frac{p}{p-1}}$, will be denoted by $\langle\cdot,\cdot\rangle$. 
According to \cite[Subsection 2.1]{FKS}, the sum space $\V^{\prime} + \L^{\frac{p}{p-1}}$ is a Banach space endowed with the norm
\begin{align*}
	\| \u \|_{\V^{\prime} + \L^{\frac{p}{p-1}}} &:= \inf \left\{ \|  \u_1 \|_{\V^{\prime}} + \|  \u_2 \|_{\L^{\frac{p}{p-1}}} : \u = \u_1 + \u_2, \ \u_1 \in \V^{\prime}, \ \u_2 \in \L^{\frac{p}{p-1}} \right\} \notag \\
	&= \sup \left\{ \frac{|\langle \u , \v \rangle|}{\| \v \|_{\V \cap \L^p}} : \boldsymbol{0} \neq \v \in \V \cap \L^p \right\},
\end{align*}
where $\| \v\|_{\V \cap \L^p} := \max \{ \| \v \|_{\V}, \| \v \|_{\L^p} \}$ defines a norm on $\V \cap \L^p$.  Furthermore, this norm is equivalent to both $\| \v\|_{\V} + \| \v \|_{\L^p}$ and $\sqrt{ \| \v \|_{\V}^2 + \| \v \|_{\L^p}^2 }$.

\subsection{Bilinear and nonlinear operators.}
Let us now define the trilinear, and nonlinear forms that will be employed in the weak formulation of \emph{Problem \ref{problem 3.1}}.
Let  the \emph{operator} $\A:\V\to \V^{\prime}$ be defined by
\begin{align}\label{Linear-1}
	\langle \A\u,\v \rangle =
	\int_{\Omega} \mathbb{T}(\mathbb{D}\u) : \mathbb{D}\v \d x \text{ for all } \u,\v \in \V.
\end{align}
Now, define the \emph{trilinear form} $b(\cdot,\cdot,\cdot):\V\times\V\times\V\to\R$ by $$b(\u,\v,\w)= \sum_{i,j=1}^{d} \int_{\Omega} y_i(x)\frac{\partial z_j(x)}{\partial x_i} w_j(x) \d x.$$
By applying H\"older's, Gagliardo-Nirenberg's and Korn’s inequalities (see \eqref{2.3}), we obtain
\begin{align}\label{eqn-bbound}
	|b(\u,\v,\w)|&\leq\|\u\|_{\L^4}\|\nabla\v\|_{\L^2}\|\w\|_{\L^4} \nonumber\\
	&\leq C_k C_g^2 \|\u\|_{\H}^{1-\frac{d}{4}}\|\u\|_{\H^1}^{\frac{d}{4}}\|\v\|_{\V}\|\w\|_{\H}^{1-\frac{d}{4}}\|\w\|_{\H^1}^{\frac{d}{4}}\nonumber\\
	&\leq C_b\|\u\|_{\V}\|\v\|_{\V}\|\w\|_{\V},
\end{align}
for all $ \u,\v,\w\in\V$ and $$C_b= C_k^3 C_g^2,$$
where $C_g$ is the constant appearing in the Gagliardo-Nirenberg's inequality and $C_k$ is defined in \eqref{2.3}. Hence, the mapping  $b(\u, \v, \cdot) $ defines a continuous linear functional on $\V$ and denoted the corresponding element in $\V^{\prime}$ by $\B(\u, \v)$ such that 
\begin{align}\label{bilinear}
	\langle\B(\u,\v),\w\rangle=b(\u,\v,\w) \ \text{ for all }\  \u,\v,\w\in\V.
\end{align}
Note that $\B(\cdot,\cdot)$ is a bilinear operator, and we write $\B(\u) = \B(\u, \u)$. Moreover, we have 
\begin{align}
	b(\u,\v,\w)&=-b(\u,\w,\v) \ \text{ for all }\  \u,\v,\w\in\V, \nonumber \\
	b(\u,\v,\v)&=0 \ \text{ for all }\  \u,\v\in\V \label{eqn-b-est-1}. 
\end{align}
Let us now define the \emph{nonlinear form} $c(\cdot, \cdot):\L^{r+1} \times \L^{r+1} \to \R$ as
$$ c(\u, \v)= \int_{\Omega} |\u|^{r-1}\u \cdot \v \d x ,$$
which is bounded, since it satisfies the estimate
$$|c(\u, \v)| \leq \|\u\|_{\L^{r+1}}^r \|\v\|_{\L^{r+1}}.$$
Furthermore, as shown in {\cite[Section 2.4]{MTMS}}, for all	$\u,\v\in\L^{r+1}$ with $r\geq 1$, we obtain 
	\begin{align}\label{2.23}
			&\langle\u|\u|^{r-1}-\v|\v|^{r-1},\u-\v\rangle\geq \frac{1}{2}\||\u|^{\frac{r-1}{2}}(\u-\v)\|_{\H}^2+\frac{1}{2}\||\v|^{\frac{r-1}{2}}(\u-\v)\|_{\H}^2\geq 0.
 		\end{align}
%		and 
%	\begin{align*}
%		&\langle\u|\u|^{r-1}-\v|\v|^{r-1},\u-\v\rangle
%		\geq\frac{1}{2^{r-1}}\|\u-\v\|_{\L^{r+1}}^{r+1}.
%	\end{align*}
Similarily, we define 
$$ \tilde{c}(\u, \v)= \int_{\Omega} |\u|^{q-1}\u \cdot \v \d x.$$
Accordingly, the associated nonlinear operators $\C(\cdot),\widetilde\C(\cdot): \L^{r+1} \to \L^{\frac{r+1}{r}}$ are defined by
\begin{align}\label{Nonlinear}
	\langle\C(\u),\v\rangle=  c(\u, \v), \quad \langle \widetilde\C(\u),\v\rangle= \tilde{c}(\u, \v) \ \text{ for all }\  \v \in  \L^{r+1}.
\end{align}

%%%%%%%%%%%%%%%%%%%%%%%%%%%%%%%%%%%%
%%%%%%%%%%%%%%%%%%%%%%%%%%%%%%%%%%%%
%%%%%%%%%%%%%%%%%%%%%%%%%%%%%%%%%%%%
\section{Variational Formulation}\label{sec3}\setcounter{equation}{0}
We now proceed to derive the weak formulation of problem \eqref{Main eq}-\eqref{bc2}, incorporating the implicit constraint $\mathbb{U}: \V\cap\L^{r+1} \multimap \V\cap\L^{r+1} $ defined by
\begin{align}\label{2.20}
	\mathbb{U}(\u):=\{\v \in  \V\cap\L^{r+1} ~|~ h(\v) \leq m(\u)\} \ \text{ for } \ \u \in  \V\cap\L^{r+1},
\end{align}
where, the functions $h$ and $m$ satisfies Hypothesis \ref{hyp-rm}. The steady-state flow problem can be classically formulated as follows:
\begin{problem}\label{problem-1}
	Find a flow velocity $\u:\Omega\to\mathbb{R}^d$, an extra stress tensor $\mathbb{S}: \mathbb{M}^d \to\mathbb{M}^d$, and a pressure $p:\Omega\to\mathbb{R}$ such that $\u\in \mathbb{U}(\u)$ and \eqref{Main eq}-\eqref{bc2} are satisfied. 
\end{problem}
Assume that  \emph{Problem \ref{problem-1}} admits sufficiently
smooth functions $\u, \mathbb{S} \text{ and } p$. Consider $\v \in \U(\u) \subset \V\cap\L^{r+1}$ and taking the inner product of equation \eqref{Main eq} with an arbitrary
smooth test function $\v-\u$ to obtain
\begin{align}\label{var-1}
	&\int_{\Omega} (-\Div \mathbb{S}) \cdot (\v-\u)\d x
	+ \int_{\Omega} \Div (\u \otimes \u) \cdot (\v-\u)\d x + \alpha	\int_{\Omega} \u \cdot (\v-\u) \d x \nonumber \\
	&\quad+ \beta	\int_{\Omega} |\u|^{r-1}\u \cdot (\v-\u) \d x+\kappa	\int_{\Omega} |\u|^{q-1}\u \cdot (\v-\u) \d x
	+ \int_{\Omega} \nabla p \cdot (\v-\u)\d x
\nonumber\\&	= \int_{\Omega} \f \cdot (\v-\u)\d x.
\end{align}
We denote the corresponding terms on the left hand side
of equation \eqref{var-1} by $\mathcal{I}_i$, $i=1,\ldots,6$. By applying Green's formula (see \cite[Theorem 2.25]{SMAOMS}), we calculate $\mathcal{I}_1$ and $\mathcal{I}_2$ in the following way:
%\begin{align*}
%	\int_{\Omega} \bigl( \boldsymbol{\sigma} : \boldsymbol{\varepsilon}(\v) + \Div \boldsymbol{\sigma} \cdot \v \bigr)\d x
%	= \int_{\Gamma} \boldsymbol{\sigma} \boldsymbol{\nu} \cdot \v \d \Gamma,
%	\quad
%	\v \in \mathrm{H}^1(\Omega;\mathbb{R}^d),\; \boldsymbol{\sigma} \in \mathrm{H}^1(\Omega;\mathbb{M}^d).
%\end{align*}
%Thus, we obtain
\begin{align}\label{3.4}
	\mathcal{I}_1
	= \int_{\Omega} \mathbb{S} : \mathbb{D}(\v-\u)\d x
	- \int_{\Gamma} (\mathbb{S}\boldsymbol{\nu} )\cdot (\v-\u)\d \Gamma,
\end{align}
and 	
\begin{align*}
	\mathcal{I}_2
	= - \int_{\Omega} (\u \otimes \u) : \mathbb{D}(\v-\u)\d x
	+ \int_{\Gamma} (\u \otimes \u)\boldsymbol{\nu}  \cdot (\v-\u)\d\Gamma.
\end{align*}
Now, by employing the relation $$ (\u \otimes \u)\boldsymbol{\nu} \cdot (\v-\u) = (\u \cdot (\v-\u))y_\nu \ \text{ on }\  \Gamma$$ in $\mathcal{I}_2$, we have
\begin{align}\label{2.44}
	\mathcal{I}_2
	&=  - \int_{\Omega} (\u \otimes \u) : \mathbb{D}(\v-\u)\d x
+ \int_{\Gamma} (\u \cdot (\v-\u))y_\nu\d\Gamma .
\end{align}
Similarily, by using  Green's formula (see \cite[Theorem 2.24]{SMAOMS}),
%which is written as follows:
%\begin{align*}
%		\int_{\Omega} \bigl( u \nabla\cdot \v + \nabla u \cdot \v \bigr)\d x
%	&= \int_{\Gamma} u (\v \cdot \boldsymbol{\nu})\d \Gamma,
%	\quad
%	u \in \mathrm{H}^1(\Omega),\; \v \in \mathrm{H}^1(\Omega;\mathbb{R}^d),
%\end{align*}
we calculate $\mathcal{I}_6$ as 
\begin{align*}
\mathcal{I}_6= - \int_{\Omega} p \nabla\cdot(\v-\u)\d x
+ \int_{\Gamma_0 \cup \Gamma_1} (z_\nu - y_\nu)\, p \d\Gamma .
\end{align*}
By using the properties
$\nabla\cdot\u = \nabla\cdot \v = 0$ in $\Omega$, $\u=\v=\boldsymbol{0}$ on $\Gamma_0$, and
$y_\nu = z_\nu = 0$ on $\Gamma_1$ in $\mathcal{I}_6$, we get 
\begin{align}\label{2.27}
	\mathcal{I}_6 = 0.
\end{align}
By combining the decomposition formula from \cite[(6.33)]{SMAOMS} with \eqref{1.6}, we infer 
\begin{align}\label{2.45}
\int_{\Gamma} (\mathbb{S}\boldsymbol{\nu})\cdot (\v-\u)\d \Gamma
=\int_{\Gamma_1} \boldsymbol{\tau}_\tau(\u)\cdot (\v_\tau - \u_\tau)\d\Gamma .
\end{align}
Thus, by using \eqref{2.45}, $\mathcal{I}_1$ in \eqref{3.4} becomes 
\begin{align}\label{2.46}
\mathcal{I}_1
	&= \int_{\Omega} \mathbb{S} : \mathbb{D}(\v-\u)\d x
	- \int_{\Gamma_1}\boldsymbol{\tau}_\tau(\u)\cdot (\v_\tau - \u_\tau)\d\Gamma.
\end{align}
On the other hand, from Appendix \ref{Appendix} together with the condition $\nabla\cdot\u =0$ in $\Omega$, we obtain 
\begin{align}\label{2.29}
b(\u,\u,\v-\u)= -\int_{\Omega} (\u \otimes \u) : \mathbb{D}(\v-\u) \d x+ \int_{\Gamma} (\u\cdot (\v-\u)) y_\nu \d\Gamma.
\end{align}
Thus, by comparing \eqref{2.44} and \eqref{2.29}, we deduce 
\begin{align}\label{2.30}
	\mathcal{I}_2=b(\u,\u,\v-\u).
\end{align}
Hence, by using \eqref{2.46}, \eqref{2.30} and \eqref{2.27} in \eqref{var-1}, we arrive at 
\begin{align}\label{var-2}
&\int_{\Omega} \mathbb{S} : \mathbb{D}(\v-\u)\d x- \int_{\Gamma_1}\boldsymbol{\tau}_\tau(\u)\cdot (\v_\tau - \u_\tau)\d\Gamma +b(\u,\u,\v-\u) + \alpha	\int_{\Omega} \u \cdot (\v-\u) \d x \nonumber \\
&\quad+ \beta	\int_{\Omega} |\u|^{r-1}\u \cdot (\v-\u) \d x+ \kappa	\int_{\Omega} |\u|^{q-1}\u \cdot (\v-\u) \d x
= \int_{\Omega} \f \cdot (\v-\u)\d x.
\end{align}
For calculating the first term in \eqref{var-2}, let us now consider
$$
\Omega_{+} := \{ x \in \Omega \mid \|\mathbb{D}\u\|_{\mathbb{M}^d} > 0 \}
\quad \text{ and } \quad
\Omega_{0} := \{ x \in \Omega \mid \|\mathbb{D}\u\|_{\mathbb{M}^d} = 0 \}.$$
By using condition \eqref{Ten-1} and the Cauchy-Schwarz inequality, we calculate the integral of $\mathbb{S} : \mathbb{D}(\v-\u)$ over $\Omega_{+}$  as following:
\begin{align}\label{2.33}
&	\int_{\Omega_{+}} \mathbb{S} : \mathbb{D}(\v-\u)\d x
\nonumber\\	&= \int_{\Omega_{+}}
	\left(
	\mathbb{T}(\mathbb{D}\u)
	+ g \frac{\mathbb{D}\u}{\|\mathbb{D}\u\|_{\mathbb{M}^d}}
	\right): \mathbb{D}(\v-\u)\d x \nonumber \\
	&= \int_{\Omega_{+}} \mathbb{T}(\mathbb{D}\u) : \mathbb{D}(\v-\u)\d x
	+ \int_{\Omega_{+}} g \frac{\mathbb{D}\u}{\|\mathbb{D}\u\|_{\mathbb{M}^d}}
	: \mathbb{D}\v \d x
	- \int_{\Omega_{+}} g \|\mathbb{D}\u\|_{\mathbb{M}^d}\d x
	\nonumber \\
	&\le \int_{\Omega} \mathbb{T}(\mathbb{D}\u) : \mathbb{D}(\v-\u)\d x
	+ \int_{\Omega_{+}} g \|\mathbb{D}\v\|_{\mathbb{M}^d}\d x
	- \int_{\Omega} g \|\mathbb{D}\u\|_{\mathbb{M}^d}\d x,
\end{align}
where we have used (T1) of Hypothesis \ref{hyp-T} also. By using condition \eqref{Ten-1}, that is, $\|\mathbb{S}\|_{\mathbb{M}^d} \le g$ in $\Omega_{0}$ and the Cauchy-Schwarz inequality, we calculate the integral of $\mathbb{S} : \mathbb{D}(\v-\u)$ over $\Omega_{0}$  as follows:
\begin{align}\label{2.34}
	\int_{\Omega_{0}} \mathbb{S} : \mathbb{D}(\v-\u)\d x
	&= \int_{\Omega_{0}} \mathbb{S} : \mathbb{D}\v\d x-\int_{\Omega_{0}} \mathbb{S} : \mathbb{D}\u\d x \nonumber \\
	&\le \int_{\Omega_{0}} \|\mathbb{S}\|_{\mathbb{M}^d} \|\mathbb{D}\v\|_{\mathbb{M}^d}\d x
	\le \int_{\Omega_{0}} g \|\mathbb{D}\v\|_{\mathbb{M}^d}\d x.
\end{align}
Thus, by using  inequalities \eqref{2.33} and \eqref{2.34}, we deduce
\begin{align}\label{2.36}
	\int_{\Omega} \mathbb{S} : \mathbb{D}(\v-\u)\d x&=\int_{\Omega_{+}} \mathbb{S} : \mathbb{D}(\v-\u)\d x+\int_{\Omega_{0}} \mathbb{S} : \mathbb{D}(\v-\u)\d x \nonumber \\
	&\le \int_{\Omega} \mathbb{T}(\mathbb{D}\u) : \mathbb{D}(\v-\u)\d x+ \int_{\Omega_{+}} g  \|\mathbb{D}\v\|_{\mathbb{M}^d} \d x\nonumber\\&\quad+ \int_{\Omega_{0}} g \|\mathbb{D}\v\|_{\mathbb{M}^d} \d x- \int_{\Omega} g \|\mathbb{D}\u\|_{\mathbb{M}^d}\d x \nonumber \\
	&\le \int_{\Omega} \mathbb{T}(\mathbb{D}\u) : \mathbb{D}(\v-\u)\d x
	+ \int_{\Omega} g \bigl( \|\mathbb{D}\v\|_{\mathbb{M}^d} - \|\mathbb{D}\u\|_{\mathbb{M}^d} \bigr)\d x.
\end{align}
From the condition given in \eqref{bc2}, namely
$$-\boldsymbol{\tau}_\tau(\u)\in k(\u_\tau) \partial j_\tau(\u_\tau) \ \text{ on }\  \Gamma_1,$$
and the definition of the Clarke subgradient in \eqref{subgradient} together yield 
\begin{align}\label{2.37}
-\boldsymbol{\tau}_\tau(\u)\cdot(\v_\tau-\u_\tau)\leq k(\u_\tau)j_\tau^0(\u_\tau;\v_\tau-\u_\tau) \ \text{ on }\  \Gamma_1.
\end{align}
Hence, by using  inequalities \eqref{2.36} and \eqref{2.37} in \eqref{var-2}, we find 
\begin{align*}
&\int_{\Omega} \mathbb{T}(\mathbb{D}\u) : \mathbb{D}(\v-\u)\d x
+b(\u,\u,\v-\u)+ \alpha	\int_{\Omega} \u \cdot (\v-\u) \d x
+ \beta	\int_{\Omega} |\u|^{r-1}\u \cdot (\v-\u) \d x \nonumber \\
&\quad+\kappa	\int_{\Omega} |\u|^{q-1}\u \cdot (\v-\u) \d x+\int_{\Gamma_1}k(\u_\tau)j_\tau^0(\u_\tau;\v_\tau-\u_\tau) \d \Gamma  +\int_{\Omega} g \bigl( \|\mathbb{D}\v\|_{\mathbb{M}^d} - \|\mathbb{D}\u\|_{\mathbb{M}^d} \bigr)\d x \nonumber \\
& 
\geq \int_{\Omega} \f \cdot (\v-\u)\d x.
\end{align*} 
Consequently, for $\f \in \V^{\prime}$ and any smooth function $\v-\u \in \V\cap\L^{r+1}$, it follows that
\begin{align*}
	&\int_{\Omega} \mathbb{T}(\mathbb{D}\u) : \mathbb{D}(\v-\u)\d x
	+b(\u,\u,\v-\u)+ \alpha	\int_{\Omega} \u \cdot (\v-\u) \d x
	+ \beta	\int_{\Omega} |\u|^{r-1}\u \cdot (\v-\u) \d x \nonumber \\
	&\quad+\kappa	\int_{\Omega} |\u|^{q-1}\u \cdot (\v-\u) \d x+\int_{\Gamma_1}k(\u_\tau)j_\tau^0(\u_\tau;\v_\tau-\u_\tau) \d \Gamma  +\int_{\Omega} g \bigl( \|\mathbb{D}\v\|_{\mathbb{M}^d} - \|\mathbb{D}\u\|_{\mathbb{M}^d} \bigr)\d x \nonumber \\
	& 
	\geq \langle \f, \v-\u \rangle,
\end{align*}
which is the required variational formulation of \emph{Problem \ref{problem-1}}. 

In summary, the preceding analysis leads to the following quasi-variational--hemivariational inequality associated with \emph{Problem \ref{problem-1}} in which the unknown variable $p$ and the extra stress have been eliminated.  
\begin{problem}\label{problem 3.1}
Find a flow velocity $\u \in \mathbb{U}(\u) \subset \V \cap \L^{r+1}$ such that	
\begin{align*}
	&\int_{\Omega} \mathbb{T}(\mathbb{D}\u) : \mathbb{D}(\v-\u)\d x
	+b(\u,\u,\v-\u)+ \alpha	\int_{\Omega} \u \cdot (\v-\u) \d x
	+ \beta	\int_{\Omega} |\u|^{r-1}\u \cdot (\v-\u) \d x\nonumber \\
	&\quad+ \kappa	\int_{\Omega} |\u|^{q-1}\u \cdot (\v-\u) \d x+\int_{\Gamma_1}k(\u_\tau)j_\tau^0(\u_\tau;\v_\tau-\u_\tau) \d \Gamma +\int_{\Omega} g \bigl( \|\mathbb{D}\v\|_{\mathbb{M}^d} - \|\mathbb{D}\u\|_{\mathbb{M}^d} \bigr)\d x
\nonumber\\ & 	\geq \langle \f, \v-\u \rangle,
\end{align*} 
for all $\v \in \mathbb{U}(\u).$
\end{problem}
%%%%%%%%%%%%%%%%%%%%%%%%%%%%%%%%%%%
%%%%%%%%%%%%%%%%%%%%%%%%%%%%%%%%%%%
%%%%%%%%%%%%%%%%%%%%%%%%%%%%%%%%%%%
\begin{remark}
As discussed in \cite{MSSD1}, we conjecture that every smooth solution of \emph{Problem \ref{problem 3.1}} is also a solution of \emph{Problem \ref{problem-1}}. It is still an intriguing open problem to recover the corresponding pressure and extra stress tensor and to interpret the slip boundary condition in \eqref{bc2} from the weak formulation.
\end{remark}
For our analysis, we impose the following assumptions on the functions 
$h$ and $m$ appearing in the implicit constraint, as introduced in \cite{MSSD1}.
\begin{hypothesis}\label{hyp-rm} 
	The functions $h : \V \cap \L^{r+1} \to \mathbb{R} \text{ and } m : \H \to \mathbb{R}$ satisfy the following assumptions:
	\begin{enumerate}
		\item[(i)] $h$ is lower semicontinuous on $\V\cap\L^{r+1}$, positively homogeneous, and convex.
		\item[(ii)] $m $ is continuous, $m_0:=\inf\limits_{\v \in \H} m(\v)>0,$ and $h(\boldsymbol{0}) \leq m_0$.
	\end{enumerate}
\end{hypothesis}
Typical examples of functions satisfying Hypothesis \ref{hyp-rm} are written as follows. For $h:\V \cap \L^{r+1}\to\mathbb{R}$, one may consider the rate dissipation energy (or drag) functional
$$
h(\v)=\frac{\delta_0}{2}\int_{\Omega}\|\mathbb{D}\v\|_{\mathbb{M}^d}^2\d x,$$
where $\delta_0>0$ denotes the viscosity coefficient. This functional quantifies the energy dissipated by viscous effects. As for $
m : \H \to \mathbb{R}$, a standard choice is 
$$
m(\v)=\mathrm{a}+\int_{\Omega}\|\v(x)\|_{\mathbb{R}^d}\vartheta(x)\d x,$$
where $\vartheta \in \H$, $\vartheta \geq 0$, and $\mathrm{a} >0$. This choice clearly ensures the positivity condition required in Hypothesis \ref{hyp-rm}.

\begin{lemma}\label{lemma-1}
Under Hypothesis \ref{hyp-rm},  the set-valued map $\mathbb{U}: \V\cap\L^{r+1} \multimap \V\cap\L^{r+1} $ defined in \eqref{2.20} has closed and convex values. In addition,  $\boldsymbol{0} \in \mathbb{U}(\u)$ for all $\u \in \V\cap\L^{r+1}$ and it is weakly Mosco continuous, that is, for any sequence $\{\boldsymbol{\upsilon}_n\}_{n\in\N} \subset \V\cap\L^{r+1}$ such that 
$	\boldsymbol{\upsilon}_n \xrightarrow{w} \boldsymbol{\upsilon} \ \text{ in }\  \V\cap\L^{r+1}
$
implies 
$\mathbb{U}(\boldsymbol{\upsilon}_n) \xrightarrow{M} \mathbb{U}(\boldsymbol{\upsilon})$. 
\end{lemma}
\begin{proof} The proof is carried out in the following steps: 
	\vskip 0.1 cm
\noindent \textbf{Step 1:}	\emph{The set $\mathbb{U}(\u)$ is a nonempty, closed and convex subset of $\V\cap\L^{r+1}$.}	Let us consider the set-valued map $\mathbb{U}: \V\cap\L^{r+1} \multimap \V\cap\L^{r+1} $ defined by
	\begin{align*}
		\mathbb{U}(\u):=\{\v \in  \V\cap \L^{r+1} ~|~ h(\v) \leq m(\u)\}.
	\end{align*}
By using  condition (ii) of Hypothesis \ref{hyp-rm}, we have $$h(\boldsymbol{0}) \leq m_0=\inf\limits_{\u \in \H} m(\u)\leq m(\u),$$
	which implies $\boldsymbol{0} \in \mathbb{U}(\u)$ for all $\u\in \V\cap\L^{r+1}.$ 
	
	In order to prove the closedness of the set $\mathbb{U}(\u)$, consider a sequence $\{\v_n\}_{n \in \N} \subset \mathbb{U}(\u)$ such that 
	$ \v_n \to \v \ \text{ as }\ n \to \infty.$
	By using condition (i) of Hypothesis \ref{hyp-rm}, that is, the lower semicontinuity of $h$, we have
	$$ h(\v) \leq \liminf_{n\to\infty} h(\v_n) \leq m(\u),$$
	which implies $\v \in \mathbb{U}(\u)$ for all $\u \in \V\cap\L^{r+1}.$
	
	To establish the convexity of $\mathbb{U}(\u)$, for any $\u\in \V\cap\L^{r+1}$, consider $\lambda \in [0,1]$ and take arbitrary 
	$\v_1, \v_2 \in \mathbb{U}(\u)$. By using condition (i) of Hypothesis \ref{hyp-rm}, that is, convexity of $h$ implies
	$$ h(\lambda \v_1 + (1-\lambda)\v_2)
	\leq \lambda h(\v_1) + (1-\lambda) h(\v_2)
	\leq \lambda m(\u) + (1-\lambda)m(\u) = m(\u),$$
	which implies $\lambda \v_1+ (1-\lambda) \v_2 \in \mathbb{U}(\u)$. Hence, $\mathbb{U}(\u)$ is a convex set.
	\vskip 0.2cm
	\noindent
	\emph{$\mathbb{U}$ is weakly Mosco continuous.} In order to prove this, for any sequence $\{\boldsymbol{\upsilon}_n\}_{n\in\N} \subset \V\cap\L^{r+1}$ such that 
\begin{align}\label{3.151}
		\boldsymbol{\upsilon}_n \xrightarrow{w} \boldsymbol{\upsilon} \ \text{ in }\  \V\cap\L^{r+1}
	\end{align}
we need to show that 
	 $\mathbb{U}(\boldsymbol{\upsilon}_n) \xrightarrow{M} \mathbb{U}(\boldsymbol{\upsilon})$ (see Definition \ref{def:mosco}). To verify condition (M$_1$) in Definition \ref{def:mosco}, for any $\v_n \in \mathbb{U}(\boldsymbol{\upsilon}_n)$ with 
	\begin{align}\label{con-v}
		\v_n \xrightarrow{w} \v \ \text{ in }\  \V\cap\L^{r+1},
	\end{align}  
	we aim to show $\v \in \mathbb{U}(\boldsymbol{\upsilon})$. Since the embedding $\V\hookrightarrow\H$ is compact and using the convergence \eqref{3.151}, there exists a subsequence of $\{\boldsymbol{\upsilon}_n\}_{n\in\N}$ (still represented by the same symbol) such that 
	\begin{align}\label{eqn-strong11}
		\boldsymbol{\upsilon}_n\to \boldsymbol{\upsilon}\ \ \text{ in }\ \ \H.
	\end{align}
Hence, by using the fact that $\v_n \in \mathbb{U}(\boldsymbol{\upsilon}_n)$, weakly lower semicontinuity of $h$ and continuity of $m$ (Hypothesis \ref{hyp-rm} (ii)) with convergences \eqref{con-v} and \eqref{eqn-strong11}, we obtain
	$$ h(\v)\leq \liminf_{n\to\infty} h(\v_n) \leq  \liminf_{n\to\infty} m(\boldsymbol{\upsilon}_n)=m(\boldsymbol{\upsilon}).$$
	Thus $\v \in \mathbb{U}(\boldsymbol{\upsilon})$, which implies $\mathbb{U}(\boldsymbol{\upsilon})$ satisfies condition (M$_1$).
	
Let us now verify condition (M$_2$) in Definition \ref{def:mosco}. We consider arbitrary $\v \in  \mathbb{U}(\boldsymbol{\upsilon})$ and set
	$$\v_n = \frac{m(\boldsymbol{\upsilon}_n)}{m(\boldsymbol{\upsilon})}\v.$$
	Then, by using Hypothesis \ref{hyp-rm}, that is, positive homogeneity of $h$ together with the assumption
		$m_0 > 0$, we obtain
	$$ h(\v_n) = \frac{m(\boldsymbol{\upsilon}_n)}{m(\boldsymbol{\upsilon})}h(\v) \leq m(\boldsymbol{\upsilon}_n),$$
	%Since $\z \in  \mathbb{U}(\v_0)$ implies $r(\z) \leq m(\v_0)$ implies $\frac{r(\z)}m(\v_0)} leq 1 .$
which implies that $\v_n \in  \mathbb{U}(\boldsymbol{\upsilon}_n)$ for every $n \in \mathbb{N}$. By using the convergence in \eqref{eqn-strong11} and continuity of $m$, we further have 
$$ \lim_{n\to \infty} \|\v_n - \v\|_{\V\cap\L^{r+1}}
= \lim_{n\to \infty} \left\| \frac{m(\boldsymbol{\upsilon}_n)}{m(\boldsymbol{\upsilon})}\v-\v \right\|_{\V\cap\L^{r+1}}
= \lim_{n\to \infty} \left| \frac{m(\boldsymbol{\upsilon}_n) - m(\boldsymbol{\upsilon})}{m(\boldsymbol{\upsilon})} \right| \|\v\|_{\V\cap\L^{r+1}}= 0. $$
Consequently, we get the convergence 
\begin{align*}
	\v_n \to \v \ \text{ in }\  \V\cap\L^{r+1} \ \text{ as }\ n \to \infty,
\end{align*}
which implies $\mathbb{U}(\boldsymbol{\upsilon})$ satisfies condition (M$_2$) in Definition \ref{def:mosco}.
\end{proof}
In order to show the existence of a weak solution of \emph{Problem \ref{problem 3.1}}, we impose the following assumptions, which are adopted from \cite{MSSD1}.
\begin{hypothesis}\label{hyp-T}
	The function $\mathbb{T} : \Omega \times \mathbb{M}^d \to \mathbb{M}^d$  satisfies the following assumptions:
	\begin{enumerate}
		\item[(T1)] $\mathbb{T}(x,\cdot)$ is continuous on $\mathbb{M}^d$ for a.e.\ $x \in  \Omega $, and 
		$\mathbb{T}(x,\boldsymbol{0})=\boldsymbol{0}$ for a.e.\ $x \in  \Omega $;
		\item[(T2)] $\mathbb{T}(\cdot,\mathbb{F})$ is measurable on $\Omega$ for all $\mathbb{F} \in \mathbb{M}^d$;
		\item[(T3)] $\|\mathbb{T}(x,\mathbb{F})\|_{\mathbb{M}^d} \leq a_1(x) + a_2 \|\mathbb{F}\|_{\mathbb{M}^d}$ for all $\mathbb{F} \in \mathbb{M}^d$, a.e.\ $ x \in  \Omega $,
		with $a_1 \in \mathrm{L}^2(\Omega)$, $a_2 > 0$;
		\item[(T4)] $(\mathbb{T}(x,\mathbb{F}_1)-\mathbb{T}(x,\mathbb{F}_2)) : (\mathbb{F}_1-\mathbb{F}_2) 
		\geq \gamma \|\mathbb{F}_1-\mathbb{F}_2\|_{\mathbb{M}^d}^2$ for all $\mathbb{F}_1,\mathbb{F}_2 \in \mathbb{M}^d$, a.e.\ $x \in \Omega$,
		with $\gamma > 0$.
	\end{enumerate}
	\end{hypothesis}

	\begin{example}[Example satisfying Hypothesis \ref{hyp-T}]\label{exmaple-1}
	The assumptions $(T1)-(T4)$ are satisfied by several constitutive laws of practical interest. As discussed in \cite{MSD1}, mathematical models involving the constitutive function $\mathbb{T}$ are of the form
	\begin{equation}\label{eq-new}
		\mathbb{T}(x,\mathbb{F})
		=\mu(\|\mathbb{F}\|_{\mathbb{M}^d})\mathbb{F},
		\  \text{ for } \ \mathbb{F}\in\mathbb{M}^d,\ \text{ a.e. } x\in\Omega,
	\end{equation}
	where $\mu:[0,\infty)\to\mathbb{R}$ is a prescribed viscosity function. Fluids governed by constitutive laws of the form \eqref{eq-new} are referred to as generalized Newtonian fluids.
	In particular, if
$	\mu(r)=\mu_0,\  r\ge0,$
	where $\mu_0>0$ is a constant viscosity coefficient, then \eqref{eq-new} reduces to
$	\mathbb{T}(x,\mathbb{F})=\mu_0\mathbb{F},$
	which corresponds to the linear constitutive law for the classical Newtonian fluid and clearly satisfies  $(T1)-(T4)$. Furthermore, the constitutive law \eqref{eq-new} reduces to the Bingham model when $\mu(r)=\mu_0$ for $r\ge0$, and to the CBFeD system in the case $g=0$.

	\end{example}

\begin{hypothesis}\label{hyp-new-j}
	The functional $j_{\tau}: \Gamma_1 \times \mathbb{R}^d \to\mathbb{R}$ satisfies the following assumptions: 
	\begin{enumerate}
		\item[(H1)] $j_{\tau}(\cdot, \boldsymbol{\xi})$ is measurable on $ \Gamma_1 \ \text{ for all }\  \boldsymbol{\xi} \in \mathbb{R}^d$ and $j_{\tau}(x,\cdot)$ is locally Lipschitz on $ \mathbb{R}^d  \text{ for a.e. } x \in \Gamma_1$;
		\item[(H2)] $\partial j_{\tau} $ satisfies the growth condition $$\|\boldsymbol{\eta}\|\leq c_0(x)+c_1\|\boldsymbol{\xi}\|_{\mathbb{R}^d},$$ for all $\boldsymbol{\eta}\in \partial j_{\tau}(x,\boldsymbol{\xi})$,  for all $\boldsymbol{\xi}\in \mathbb{R}^d$, a.e. $x \in \Gamma_1$ with $c_0 \in \mathrm{L}^2(\Gamma_1),~ c_0,c_1 \geq 0$;
		\item[(H3)]  $j_{\tau}(x,\cdot)$ or $-j_{\tau}(x,\cdot)$ is \emph{regular} for a.e. $x \in \Gamma_1$.
	\end{enumerate}
\end{hypothesis}
\begin{hypothesis}\label{hyp-k}
The functional $k: \Gamma_1 \times \mathbb{R}^d \to\mathbb{R}$ satisfies the following assumptions:
	\begin{enumerate}
\item[(i)]  $k(\cdot, \boldsymbol{\xi})$ is measurable on $ \Gamma_1 \ \text{ for all }\  \boldsymbol{\xi} \in \mathbb{R}^d$;
\item[(ii)] $k(x,\cdot)$ is continuous on $ \mathbb{R}^d  \text{ for a.e. } x \in \Gamma_1$;
\item[(iii)] $0<k_0 \leq k(x,\boldsymbol{\xi}) \leq k_1$ for all $\boldsymbol{\xi} \in \mathbb{R}^d, \text{ a.e. } x \in \Gamma_1$ for some $k_0,k_1 \in \mathbb{R}$.
	\end{enumerate}
\end{hypothesis}

\begin{example}
A particular example satisfying Hypotheses \ref{hyp-new-j} and \ref{hyp-k} is the nonlinear Navier--Fujita slip condition, as discussed in \cite{MSD1}, obtained by choosing
$$
k(x,\boldsymbol{\xi})=\phi_1(x)+\phi_2(x,\|\boldsymbol{\xi}\|_{\mathbb{R}^d}), \qquad
j_{\tau}(x,\boldsymbol{\xi})=\|\boldsymbol{\xi}\|_{\mathbb{R}^d},$$
where $\phi_1:\Gamma_1\to(0,\infty)$ satisfies
$$
\phi_1 \in \mathrm{L}^{\infty}(\Gamma_1), \
0<k_0\le \phi_1(x)\leq M, \ \text{ a.e. } \ x\in\Gamma_1 $$
and $\phi_2:\Gamma_1\times[0,\infty)\to[0,\infty)$ is a Carath\'eodory function (\cite[Lemma 2.4]{FR-2018}) satisfying
$$
0 < \phi_2(x,s)\le M_1 \ \text{ for all } \ s > 0,\ \text{ a.e. } x\in\Gamma_1,$$
and
$$ \phi_2(x,s)=0  \text{ if and only if } s=0,  \text{ a.e. } x\in\Gamma_1.$$
In this case, $j_{\tau}$ satisfies Hypothesis \ref{hyp-new-j} with
$c_0(x)=1$ and $c_1=0$. This formulation encompasses several classical slip laws, including the Navier slip condition (\cite{CLN}), nonlinear Navier-type slip conditions for non-Newtonian fluids (\cite{LRC}), and threshold friction-type slip conditions (\cite{HFU}).
\end{example}

Let us define a functional $J:\L^2(\Gamma_1) \times \L^2(\Gamma_1) \to\mathbb{R}$ by 
\begin{align}\label{eqn-new-J}
	J(\w, \u)=\int_{\Gamma_1} k(x,\w(x))j_{\tau}(x,\u(x))\d\Gamma,\ \w,\u \in \L^2(\Gamma_1),
\end{align}
and $\Phi: \V\cap\L^{r+1} \to \mathbb{R}$ by 
\begin{align}\label{eqn-phi}
\Phi(\v)= \int_{\Omega} g  \|\mathbb{D}\v\|_{\mathbb{M}^d}\d x \ \text{ for all }\  \v \in \V\cap\L^{r+1}.
\end{align}

\begin{lemma}[{\cite[Theorem 5.1]{MSSD1}}]\label{lem-J-lem}
		Assume that $j_\tau, k:\Gamma_1\times \mathbb{R}^d  \to\mathbb{R}$ satisfy Hypothesis \ref{hyp-new-j} and Hypothesis \ref{hyp-k}, respectively. Then the functional $J$, as defined in \eqref{eqn-new-J}, satisfies the following properties:
	\begin{enumerate} 
\item [(i)] $J(\w,\cdot)$ is locally Lipschitz for all $\w \in \L^2(\Gamma_1) $;
\item [(ii)] $\|\boldsymbol{\eta}\|_{\L^2(\Gamma_1)}\leq C_1 +C_2\|\w\|_{\L^2(\Gamma_1)}+C_3 \|\u\|_{\L^2(\Gamma_1)}$ for all $\boldsymbol{\eta}\in \partial J(\w,\u)$, $\w,\u\in \L^2(\Gamma_1)$ with $C_1=\sqrt{2}k_1\|c_0\|_{\mathrm{L}^2(\Gamma_1)},C_2=0 \text{ and } C_3=\sqrt{2}k_1 c_1$;
		\item [(iii)]  $J^0(\w,\u;\y)=\int_{\Gamma_1}k(x,\w(x))j_{\tau}^0(x,\u(x);\y(x))\d\Gamma$ for all $\w, \u, \y \in \L^2(\Gamma_1)$.
%		\item[(iv)] $J^0(\w,\u;\y)$ is upper semicontinuous for all $\w, \u,\y\in \L^2(\Gamma_1)$.
	\end{enumerate}
\end{lemma}
%%%%%%%%%%%%%%%%%%%%%%%%%%%%%%%%%%%
%%%%%%%%%%%%%%%%%%%%%%%%%%%%%%%%%%%
%%%%%%%%%%%%%%%%%%%%%%%%%%%%%%%%%%%

%and $\A_0:\V\to \V^{\prime}$ defined by
%\begin{align}\label{Linear-2}
%	\langle \A_0\u,\v \rangle =
%	\int_{\Omega}\u \cdot \v \d x, \quad \u,\v \in \V.
%\end{align}
%%%%%%%%%%%%%%%%%%%%%%%%%%%%%%%%%%%%
%%%%%%%%%%%%%%%%%%%%%%%%%%%%%%%%%%%%
%%%%%%%%%%%%%%%%%%%%%%%%%%%%%%%%%%%%

\section{The CBFeD Quasi-variational--Hemivariational Inequality}\label{sec4}\setcounter{equation}{0}
In this section, we establish the main results concerning the existence and uniqueness of weak solutions for the proposed system.  By using the operators introduced in \eqref{Linear-1}, \eqref{bilinear}, \eqref{Nonlinear}, \eqref{eqn-new-J}, and \eqref{eqn-phi}, \emph{Problem \ref{problem 3.1}} can be reformulated as the following quasi-variational--hemivariational inequality. To prove the existence of weak solutions for \emph{Problem \ref{Problem 2.17}}, we mainly follow the ideas presented in \cite{MSSD1}.

\begin{problem}\label{Problem 2.17}
Find $\u \in \V\cap\L^{r+1}$ such that $\u \in \mathbb{U}(\u)$ and
\begin{align*}
\langle \mathcal{F}(\u),\v-\u \rangle+J^0(\mathcal{L}\u,\mathcal{L}\u;\mathcal{L}\v-\mathcal{L}\u)+ \Phi(\v)-\Phi(\u) \geq \langle \f, \v-\u\rangle,
\end{align*} 
 for all $ \v \in \mathbb{U}(\u),$ where $\mathcal{F}(\cdot):\V\cap\L^{r+1}\to\V^{\prime}+\L^{\frac{r+1}{r}}$ is defined by  $$ \mathcal{F}(\u):=  \A\u+\B(\u)+\alpha \u+\beta \C(\u)+\kappa \widetilde\C(\u),$$ and the mapping $ \mathcal{L}:\V \to \L^2(\Gamma_1)$ is defined by $\mathcal{L}\u=\u_{\tau}|_{\Gamma_1}.$
\end{problem}

To establish the solvability of \emph{Problem \ref{Problem 2.17}}, we first introduce an auxiliary problem obtained by fixing the bilinear and multivalued terms. More precisely, for a fixed pair
$$(\v_0,\w_0) \in \V\cap\L^{r+1} \times \L^2(\Gamma_1)$$
we replace the convection term $\B(\u)$ by $\B(\v_0,\u)$ and treat the boundary contribution associated with the hemivariational term through the parameter $\w_0$. This leads to the intermediate variational inequality written as follows.

\begin{problem}\label{prob3.4}
	Find $\u \in \V\cap\L^{r+1}$ such that $\u \in \mathbb{U}(\v_0)$ and
	\begin{equation}\label{3.1}
		\langle \A \u + \B(\v_0,\u)+\alpha \u+\beta \C(\u)+\kappa \widetilde\C(\u)- \f+\mathcal{L}^*\w_0, \v-\u \rangle+\Phi(\v)-\Phi(\u)  \geq 0 \ \text{ for all }\  \v \in \mathbb{U}(\v_0),
	\end{equation}
where $ \mathcal{L}^*:\L^2(\Gamma_1) \to \V^{\prime} $ is the adjoint operator to $\mathcal{L}$.
\end{problem}
For ease of notation, we denote 
$$ \|\mathcal{L}\|_{\mathscr{L}(\V,\L^2(\Gamma_1))}:=\|\mathcal{L}\| \ \text{ and }\  \|\mathcal{L}^*\|_{\mathscr{L}(\L^2(\Gamma_1), \V^{\prime})}:=\|\mathcal{L}^*\|$$ in the remainder of the paper.
\begin{theorem}\label{theorem3.5}
Let   \begin{align}\label{restrict}
\frac{\gamma}{C^2_pC^2_k}+\alpha > 2 \varrho.
	\end{align} 
%	for some sufficiently small $\varepsilon>0$, 
	where 
\begin{align}\label{varrho}
	\varrho
	= \left( \frac{r-q}{r-1} \right)
	\left( \frac{4(q-1)}{\beta (r-1)} \right)^{\frac{q-1}{r-q}}
	\left( |\kappa| q 2^{q-1} \right)^{\frac{r-1}{r-q}}.
\end{align} 
Assume that $\f\in \V^{\prime}$, and Hypotheses \ref{hyp-rm}  and  \ref{hyp-T}  hold. Then  Problem \ref{prob3.4} has a unique solution.
\end{theorem}

\begin{remark}
	%Note that the condition \eqref{restrict} is equivalent to $	\frac{\gamma}{C^2_pC^2_k}+\alpha > 2 \varrho$. 
	We emphasize here that for the case of CBF equations, that is, $\kappa=0$,  the condition \eqref{restrict} holds true without any restriction on $\gamma>0$ and $\alpha\geq 0$. 
\end{remark}

\begin{proof}[Proof of Theorem \ref{theorem3.5}]
For proving the existence of a solution $\u \in \V\cap\L^{r+1}$ for \emph{Problem \ref{prob3.4}}, we consider 
$\mathcal{G}(\cdot):\V\cap\L^{r+1}\to\V^{\prime}+\L^{\frac{r+1}{r}}$ defined by 
$$\mathcal{G}(\u) := \A \u + \B(\v_0,\u)+\alpha \u+\beta \C(\u)+\kappa \widetilde\C(\u) \ \text{ for all }\  \u \in \V\cap\L^{r+1}.$$ 
Following the approach developed in \cite{MSSD2}, we investigate the existence and uniqueness of a solution to \emph{Problem \ref{prob3.4}}. The proof is carried out through a sequence of carefully structured steps written as follows:
	\vskip 0.2cm
\noindent
\textbf{Step I:} \emph{Strong monotone type property of the operator $\mathcal{G}(\cdot)$.} For proving the strong monotonicity of the operator $\mathcal{G}(\cdot)$, we estimate
\begin{align*}
	&\langle \mathcal{G}(\u_1) - \mathcal{G}(\u_2), \u_1 - \u_2 \rangle \\
	&= \langle \A \u_1 - \A \u_2, \u_1 - \u_2 \rangle
	+ \langle \B(\v_0, \u_1) - \B(\v_0, \u_2), \u_1 - \u_2 \rangle +\alpha \langle \u_1-\u_2,\u_1-\u_2 \rangle \\
	&\quad+\beta \langle\C(\u_1)-\C(\u_2),\u_1 - \u_2\rangle+\kappa \langle\widetilde\C(\u_1)-\widetilde\C(\u_2),\u_1 - \u_2\rangle.
\end{align*}
Since the operator $\B(\v_0,\cdot) : \V \to \V^{\prime}$ is linear for all
$\v_0 \in \V$, thus we obtain
\begin{align}\label{2.48}
	&\langle \mathcal{G}(\u_1) - \mathcal{G}(\u_2), \u_1 - \u_2 \rangle \nonumber\\
	&= \langle \A \u_1 - \A \u_2, \u_1 - \u_2 \rangle
	+ \langle \B(\v_0, \u_1-\u_2), \u_1 - \u_2 \rangle +\alpha \langle \u_1-\u_2,\u_1-\u_2 \rangle \nonumber \\
	&\quad+\beta \langle\C(\u_1)-\C(\u_2),\u_1 - \u_2\rangle+\kappa \langle\widetilde\C(\u_1)-\widetilde\C(\u_2),\u_1 - \u_2\rangle.
\end{align}
From (T4) of Hypothesis \ref{hyp-T}, it follows that
\begin{align*}
	\langle \A\u_1 - \A\u_2, \u - \u_2 \rangle
	%&= \int_{\Omega} \mathbb{T}(\mathbb{D}\u_1): \mathbb{D}(\u_1 - \u_2)\d x-\int_{\Omega} \mathbb{T}(\mathbb{D}\u_2): \mathbb{D}(\u_1 - \u_2)\d x  \\
	&= \int_{\Omega} \big( \mathbb{T}(\mathbb{D}\u_1) - \mathbb{T}(\mathbb{D}\u_2) \big)
	: \mathbb{D}(\u_1 - \u_2)\d x \\
	&\geq \gamma \int_{\Omega} \|\mathbb{D}(\u_1 - \u_2)\|_{\mathbb{M}^d}^2 \d x \\
	&= \gamma \|\u_1 - \u_2\|_{\V}^2 \  \text{ for all }\  \u_1, \u_2 \in \V,
\end{align*}
which implies 
	\begin{align}\label{2.60}
		\langle \A\u_1 - \A\u_2, \u_1 - \u_2 \rangle \geq \gamma \|\u_1 - \u_2\|_{\V}^2 \geq 0.
	\end{align}
By using the properties of $\B$ as defined in \eqref{eqn-b-est-1},
we get 
\begin{align}\label{estimate-b}
	 \langle \B(\v_0, \u_1-\u_2), \u_1 - \u_2 \rangle =0.
\end{align}
From the relation \eqref{2.23}, we have
\begin{align}\label{2.24}
	&\langle \C(\u_1)-\C(\u_2),\u_1-\u_2\rangle\geq \frac{1}{2}\||\u_1|^{\frac{r-1}{2}}(\u_1-\u_2)\|_{\H}^2+\frac{1}{2}\||\u_2|^{\frac{r-1}{2}}(\u_1-\u_2)\|_{\H}^2\geq 0,
\end{align}
and an application of Taylor’s formula \cite[Theorem 7.9-1]{PGC} yields
\begin{align}\label{2.49}
&	|\kappa| |\langle \widetilde\C(\u_1) - \widetilde\C(\u_2), \u_1 - \u_2 \rangle|
	\nonumber\\&= |\kappa| \bigg|\left\langle
	\int_0^1 \widetilde\C'\big(\theta \u_1 + (1-\theta)\u_2\big)\d\theta (\u_1-\u_2),\u_1-\u_2
	\right\rangle \bigg| \nonumber \\
	&\leq |\kappa| q 
	\left\langle
	\int_0^1 \big|\theta \u_1 + (1-\theta)\u_2\big|^{q-1} \d\theta |\u_1-\u_2|, |\u_1-\u_2|
	\right\rangle \nonumber \\
	&\leq |\kappa| q 2^{q-1}
	\left\langle
	\big(|\u_1|^{q-1} + |\u_2|^{q-1}\big)|\u_1-\u_2|, |\u_1-\u_2|
	\right\rangle \nonumber \\
	&= |\kappa| q 2^{q-1} \| |\u_1|^{\frac{q-1}{2}}(\u_1-\u_2) \|_{\H}^{2}
	+ |\kappa| q 2^{q-1} \| |\u_2|^{\frac{q-1}{2}}(\u_1-\u_2) \|_{\H}^{2}.
\end{align}
By using H\"older's and Young's inequalities, we estimate
$|\kappa| q 2^{q-1} \| |\u_1|^{\frac{q-1}{2}} (\u_1-\u_2) \|_{\H}^{2}$ as
\begin{align}\label{2.50}
	&|\kappa| q 2^{q-1} \| |\u_1|^{\frac{q-1}{2}} (\u_1-\u_2) \|_{\H}^{2}\nonumber \\
	&= |\kappa| q 2^{q-1}
	\int_{\Omega} |\u_1(x)|^{q-1} |\u_1(x)-\u_2(x)|^2\d x \nonumber \\
	&= |\kappa| q 2^{q-1}
	\int_{\Omega}|\u_1(x)|^{q-1}
	|\u_1(x)-\u_2(x)|^{\frac{2(q-1)}{r-1}}
	|\u_1(x)-\u_2(x)|^{\frac{2(r-q)}{r-1}} \d x \nonumber \\
	&\leq |\kappa| q 2^{q-1}
	\left(
	\int_{\Omega} |\u_1(x)|^{r-1} |\u_1(x)-\u_2(x)|^2 \d x
	\right)^{\frac{q-1}{r-1}}\left(
	\int_{\Omega} |\u_1(x)-\u_2(x)|^2\d x
	\right)^{\frac{r-q}{r-1}} \nonumber \\
	&\le \frac{\beta}{4}
	\int_{\Omega} |\u_1(x)|^{r-1} |\u_1(x)-\u_2(x)|^2 \d x
	+ \varrho
	\int_{\Omega} |\u_1(x)-\u_2(x)|^2\d x,
\end{align}
where $\varrho$ is defined in \eqref{varrho}. 
A similar calculation yields
\begin{align}\label{2.52}
&	|\kappa| q 2^{q-1}
	\big\| |\u_2|^{\frac{q-1}{2}} (\u_1-\u_2) \big\|_{\H}^{2}
\nonumber\\	&\le \frac{\beta}{4}
	\int_{\Omega} |\u_2(x)|^{r-1} |\u_1(x)-\u_2(x)|^2\d x
	+ \varrho
	\int_{\Omega} |\u_1(x)-\u_2(x)|^2\d x .
\end{align}	
Thus, using \eqref{2.50} and \eqref{2.52} in \eqref{2.49}, we deduce
\begin{align}\label{2.54}
	|\kappa| |\langle \widetilde\C(\u_1) -\widetilde\C(\u_2), \u_1 - \u_2 \rangle|
	&\le \frac{\beta}{4}
	\big\| |\u_1|^{\frac{r-1}{2}} (\u_1-\u_2) \big\|_{\H}^{2}
	+ \frac{\beta}{4}
	\big\| |\u_2|^{\frac{r-1}{2}} (\u_1-\u_2) \big\|_{\H}^{2} \nonumber \\
	&\quad + 2\varrho
	\|\u_1-\u_2\|_{\H}^{2}.
\end{align}
Using inequalities \eqref{2.60}, \eqref{estimate-b}, \eqref{2.24} and \eqref{2.54} in \eqref{2.48}, we obtain 
\begin{align}\label{Monotone}
	\langle \mathcal{G}(\u_1) - \mathcal{G}(\u_2), \u_1 - \u_2 \rangle
	&\geq \gamma \|\u_1 - \u_2\|_{\V}^2+(\alpha-2 \varrho) \|\u_1-\u_2\|_{\H}^{2}\nonumber \\
	&\quad +\frac{\beta}{4}
	\big\| |\u_1|^{\frac{r-1}{2}} (\u_1-\u_2) \big\|_{\H}^{2}
	+ \frac{\beta}{4}
	\big\| |\u_2|^{\frac{r-1}{2}} (\u_1-\u_2) \big\|_{\H}^{2}.
\end{align}
Using the fact that $(\rho_1+\rho_2)^p\leq 2^{p-1}(\rho_1^p+\rho_2^p)$ for all $\rho_1,\rho_2\geq 0$ and $1\leq p<\infty$, we find 
\begin{align}\label{estimate-1}
	\|\u_1 - \u_2\|_{\L^{r+1}}^{r+1}
	&= \int_{\Omega} |\u_1(x) - \u_2(x)|^{r-1} |\u_1(x) - \u_2(x)|^{2}\d x \nonumber \\
	&\le 2^{r-2} \int_{\Omega} \left(|\u_1(x)|^{r-1} + |\u_2(x)|^{r-1}\right)
	|\u_1(x) - \u_2(x)|^{2}\d x \nonumber \\
	&= 2^{r-2}	\big\||\u_1|^{\frac{r-1}{2}}(\u_1 - \u_2)	\big\|_{\H}^{2}
	+ 2^{r-2}	\big\||\u_2|^{\frac{r-1}{2}}(\u_1 - \u_2)	\big\|_{\H}^{2}.
\end{align}
Hence, by using   \eqref{estimate-1} and   \eqref{2p5} in \eqref{Monotone}, we derive for some sufficiently small $\varepsilon >0$ that 
\begin{align}\label{4.13}
	&\langle \mathcal{G}(\u_1) - \mathcal{G}(\u_2), \u_1 - \u_2 \rangle \nonumber\\
	&\geq \gamma \|\u_1 - \u_2\|_{\V}^2+(\alpha-2 \varrho) \|\u_1-\u_2\|_{\H}^{2}+\frac{\beta}{2^r}\|\u_1 - \u_2\|_{\L^{r+1}}^{r+1}\nonumber\\
		&\geq \gamma \varepsilon  \|\u_1 - \u_2\|_{\V}^2+\gamma (1-\varepsilon) \|\u_1 - \u_2\|_{\V}^2+(\alpha-2 \varrho) \|\u_1-\u_2\|_{\H}^{2}+\frac{\beta}{2^r}\|\u_1 - \u_2\|_{\L^{r+1}}^{r+1}\nonumber\\
	&\geq \gamma \varepsilon  \|\u_1 - \u_2\|_{\V}^2+ \left(\frac{\gamma (1-\varepsilon)}{C^2_pC^2_k}+\alpha-2\varrho\right) \|\u_1-\u_2\|_{\H}^{2}+\frac{\beta}{2^r}\|\u_1 - \u_2\|_{\L^{r+1}}^{r+1}\nonumber \\
	&\geq \omega(\|\u_1 - \u_2\|_{\V}^2+\|\u_1 - \u_2\|_{\L^{r+1}}^{r+1}),
\end{align}
where $\omega=\min\left\{ \gamma \varepsilon,\frac{\beta}{2^r}\right\}$, provided condition \eqref{restrict} is satisfied. 
Therefore, $\mathcal{G}(\cdot)$ possesses strong monotone type property for $\frac{\gamma (1-\varepsilon)}{C^2_pC^2_k}+\alpha \geq 2 \varrho$.
%\vskip 2mm
%\noindent
%\textbf{Case 2.} \emph{$d =3 $ with $r\in (5, \infty)$}. 
	\vskip 0.2cm
\noindent
\textbf{Step II:} \emph{Pseudomonotonocity of the operator $\mathcal{G}(\cdot):\V\cap \L^{r+1}\to \V^{\prime}+ \L^{\frac{r+1}{r}}$.} In order to prove the pseudomonotone
property of the operator $\mathcal{G}(\cdot)$, we first show the boundedness of $\mathcal{G}$ (Definition \ref{def-pseudo}-(1)). 
\vskip 0.2cm
\noindent
\emph{Boundedness of $\mathcal{G}$.} By using (T3) of Hypothesis \ref{hyp-T} and H\"older's inequality, we obtain
\begin{align*}
	\int_{\Omega} \mathbb{T}(\mathbb{D}\u) : \mathbb{D}\v \d x &\leq \left( \int_{\Omega}\left( a_1(x) + a_2 \|\mathbb{D}\u\|_{\mathbb{M}^d} \right)^2 \d x \right)^{1/2}
	\left( \int_{\Omega} \|\mathbb{D}\v\|_{\mathbb{M}^d}^2 \d x \right)^{1/2} \\
	&\leq \left( \int_{\Omega} 2\left( a_1^2(x) + a_2^2 \|\mathbb{D}\u\|_{\mathbb{M}^d}^2 \right)  \d x \right)^{1/2}
	\left( \int_{\Omega} \|\mathbb{D}\v\|_{\mathbb{M}^d}^2 \d x \right)^{1/2} \\
	&\leq \left( 2\|a_1\|_{\mathrm{L}^2(\Omega)}^2 + 2a_2^2 \|\u\|_{\V}^2 \right)^{1/2} \|\v\|_{\V} \\
	&= \sqrt{2}\left( \|a_1\|_{\mathrm{L}^2(\Omega)} + a_2 \|\u\|_{\V} \right)\|\v\|_{\V},
\end{align*}
for all $\u,\v \in \V$, which implies
\begin{align}\label{2.40}
	|\langle\A \u,\v \rangle| \leq \sqrt{2}\left( \|a_1\|_{\mathrm{L}^2(\Omega)} + a_2 \|\u\|_{\V} \right)\|\v\|_{\V}.
\end{align}
From  \eqref{2.40}, we infer that $\A:\V\to\V^{\prime}$ is a bounded operator. Using the bound \eqref{eqn-bbound}, we estimate $|\langle\B(\v_0,\u),\v\rangle|$ for fixed $\v_0 \in \V$ as 
\begin{align}
	|\langle\B(\v_0,\u),\v\rangle|&\leq C_b\|\v_0\|_{\V}\|\u\|_{\V}\|\v\|_{\V} \ \text{ for all }\  \u,\v\in\V.
\end{align}
Moreover, an application of H\"older's inequality yields for all
$\u,\v \in \L^{r+1}$ that
\begin{align}
	|\langle \C(\u), \v \rangle|
	\leq \|\u\|_{\L^{r+1}}^{r}\|\v\|_{\L^{r+1}},
\end{align}
and 
\begin{align}\label{non-est-1}
	|\langle \widetilde\C(\u), \v \rangle|
	\leq \|\u\|_{\L^{q+1}}^{q}\|\v\|_{\L^{q+1}}
	\leq |\Omega|^{\frac{q(r-q)}{(q+1)(r+1)}}
	\|\u\|_{\L^{r+1}}^{q}\|\v\|_{\L^{r+1}},
\end{align}
where $|\Omega|$ is the measure of $\Omega$. Therefore, by combining the estimates \eqref{2.40}-\eqref{non-est-1} for all $\u \in \V \cap \L^{r+1}$, we have
\begin{align*}
\sup_{\|\v\|_{\V\cap \L^{r+1}} \leq 1}|\langle \mathcal{G}(\u), \v \rangle| &\leq
\sqrt{2}\left(\|a_1\|_{\mathrm{L}^2(\Omega)}+ a_2 \|\u\|_{\V} \right)+ C_b\|\v_0\|_{\V}\|\u\|_{\V}\nonumber \\ 
&\qquad +C\alpha \|\u\|_{\H}+ \beta \|\u\|_{\L^{r+1}}^{r}
	+ |\kappa||\Omega|^{\frac{q(r-q)}{(q+1)(r+1)}} \|\u\|_{\L^{r+1}}^{q}< \infty,
\end{align*}
which shows the boundedness of the operator $\mathcal{G}:\V\cap \L^{r+1}\to \V^{\prime}+ \L^{\frac{r+1}{r}}.$
\vskip 0.2cm
\noindent
\emph{Pseudomonotonocity of the operator $\mathcal{G}(\cdot):\V\cap \L^{r+1}\to \V^{\prime}+ \L^{\frac{r+1}{r}}$.} For proving this, firstly  inequality \eqref{2.60} implies that the operator $\A$ is monotone. By using the Cauchy-Schwarz inequality and \cite[Theorem 1.5.2]{ZdSmP} together with Hypothesis \ref{hyp-T}, we infer that $\A$ is continuous from $\V$ to $\V^{\prime}$ which, in turn, implies that $\A$ is hemicontinuous. Moreover, by \eqref{2.40}, the operator $\A$ is bounded. Since $\A$ is also monotone and hemicontinuous, it follows from \cite[Theorem 3.69]{SMAOMS} that $\A$ is pseudomonotone.

Let us now
establish that the operator
$\widehat{\mathcal{G}}(\cdot):=\alpha\mathrm{I}+\beta\C(\cdot)+\kappa \widetilde\C(\cdot):\V\cap\L^{r+1}\to\V^{\prime}+\L^{\frac{r+1}{r}}$ is pseudomonotone. 
In order to prove this, we consider a sequence
$\{\u_n\}_{n\in\N}\in\V\cap\L^{r+1}$ such that 
\begin{align}\label{eqn-pseudo}
	\u_n\xrightarrow{w}\u\ \ \text{ in }\ \ \V\cap\L^{r+1}\ \text{ and }\ \limsup\limits_{n\to\infty} \langle\widehat{\mathcal{G}}(\u_n),\u_n-\u\rangle\leq 0. 
\end{align}
Since weak convergence implies uniform boundedness, the sequence $\{\u_n\}_{n\in\N}$ is uniformly bounded in $\V\cap\L^{r+1},$ and 
\begin{align*}
	\u_n\xrightarrow{w}\u\ \ \text{ in }\ \ \V\ \text{ and }\ \u_n\xrightarrow{w}\u\ \ \text{ in }\ \ \L^{r+1}. 
\end{align*}
Because the embedding $\V\hookrightarrow\H$ is compact, thus there exists a subsequence of $\{\u_n\}_{n\in\N}$ (still denoted by the same symbol) such that 
\begin{align}\label{eqn-strong}
	\u_n\to \u\ \ \text{ in }\ \ \H, 
\end{align}
and, possibly along a further subsequence (still denoted by the same symbol), we have the following convergence also:  
\begin{align}\label{eqn-ae}
	\u_n(x)\to\u(x) \ \text{ for a.e. }\ x\in\Omega. 
	\end{align}
	Hence, by combining \eqref{2.24} and \eqref{2.54}, we finally get 
	\begin{align*}
	\langle\widehat{\mathcal{G}}(\u)-\widehat{\mathcal{G}}(\v),\u-\v\rangle+ 2 \varrho\|\u-\v\|_{\H}^{2} \geq \alpha\|\u-\v\|_{\H}^{2} \geq 0,
	\end{align*}
	which implies that the operator $\widehat{\mathcal{G}}+ 2\varrho \mathrm{I}$ is monotone. Therefore, we have
	\begin{align*}
		\langle\widehat{\mathcal{G}}(\u_n),\u_n-\u\rangle+2 \varrho\|\u_n-\u\|_{\H}^{2}\geq \langle\widehat{\mathcal{G}}(\u),\u_n-\u\rangle.
	\end{align*}
	Taking the limit inferior on both sides and using the strong convergence given in \eqref{eqn-strong}, we arrive at 
	\begin{align*}
		\liminf_{n\to\infty}	\langle\widehat{\mathcal{G}}(\u_n),\u_n-\u\rangle\geq  	\liminf_{n\to\infty}	\langle\widehat{\mathcal{G}}(\u),\u_n-\u\rangle=0, 
	\end{align*}
	where we have used the fact that  $	\u_n\xrightarrow{w}\u\  \text{ in }\ \V\cap\L^{r+1}$. Combining this with assumption \eqref{eqn-pseudo}, that is,
	\begin{align*}
		\limsup\limits_{n\to\infty} \langle\widehat{\mathcal{G}}(\u_n),\u_n-\u\rangle\leq 0,
	\end{align*}
	we conclude that
	\begin{align*}
		\lim_{n\to\infty}\langle\widehat{\mathcal{G}}(\u_n),\u_n-\u\rangle=0. 
	\end{align*}
	Let us now show that $\widehat{\mathcal{G}}(\u_n)\xrightarrow{w}\widehat{\mathcal{G}}(\u)$ in $\V^{\prime}+\L^{\frac{r+1}{r}}$. For all $\z \in\V\cap\L^{r+1}$, we consider
	\begin{align}\label{eqn-lim}
		\langle \widehat{\mathcal{G}}(\u_n)-\widehat{\mathcal{G}}(\u),\z \rangle&=\alpha(\u_n-\u,\z)+\beta\langle\C(\u_n)-\C(\u),\z \rangle+\kappa \langle\wi\C(\u_n)-\wi\C(\u),\z \rangle . 
	\end{align}
	The first term on the right hand side of the above inequality converges to zero using \eqref{eqn-strong}.  For the nonlinear terms involving  $\C(\cdot)$  and $\wi\C(\cdot)$, we observe that from the convergence \eqref{eqn-ae} that 
	\begin{align*}
		|\u_n(x)|^{r-1}\u_n(x)\to 	|\u(x)|^{r-1}\u(x)\ \text{ for a.e. }\ x\in\Omega.
	\end{align*}
	Moreover, we know that $\|\C(\u_n)\|_{\L^{\frac{r+1}{r}}}=\|\u_n\|_{\L^{r+1}}^r\leq C$ and $\u\in\L^{r+1}$. An application of Lemma \ref{Lem-Lions} yields 
	\begin{align*}
		\C(\u_n)\xrightarrow{w}\C(\u)\ \ \text{ in }\ \ \L^{\frac{r+1}{r}}\ \text{ as }\  n\to\infty.
	\end{align*}
	A similar argument shows that
	\begin{align*}
		\wi\C(\u_n)\xrightarrow{w}\wi\C(\u)\ \ \text{ in }\ \ \L^{\frac{q+1}{q}}\ \text{ as }\  n\to\infty.
	\end{align*}
	Therefore, the last two terms in \eqref{eqn-lim} also vanishes as $n \to \infty$ giving 
	\begin{align*}
		\langle \widehat{\mathcal{G}}(\u_n)-\widehat{\mathcal{G}}(\u),\z \rangle\to 0\ \text{ as }\ n\to\infty\ \mbox{	for all  $\z \in\V\cap\L^{r+1}$. }
	\end{align*}
Thus, both conditions given in \eqref{eqn-con-pseudo} are satisfied, and hence the operator $\widehat{\mathcal{G}}(\cdot):\V\cap\L^{r+1}\to\V^{\prime}+\L^{\frac{r+1}{r}}$ is pseudomonotone. 

Let us now prove that the operator $\B(\v_0,\cdot): \V\cap\L^{r+1}\to\V^{\prime}+\L^{\frac{r+1}{r}}$ is
completely continuous. In order to establish this result, let $\{\u_n\}_{n\in\N}$ be a sequence in $\V\cap\L^{r+1}$ such that
\begin{align*}
	\u_n\xrightarrow{w}\u\ \ \text{ in }\ \ \V\cap\L^{r+1}. 
\end{align*}
Using H\"older's and Gagliardo-Nirenberg's inequalities, we estimate
\begin{align}\label{2.62}
	|\langle\B(\v_0,\u_n-\u),\v\rangle|&\leq \|\v_0\|_{\L^4}\|\nabla \v\|_{\H}\|\u_n-\u\|_{\L^4}\nonumber\\&\leq C \|\v_0\|_{\L^4}\|\u_n-\u\|_{\H}^{1-\frac{d}{4}}\|\u_n-\u\|_{\H^1}^{\frac{d}{4}}\|\v\|_{\V}\nonumber\\&\leq C\|\v_0\|_{\V}\left(\|\u_n\|_{\V}^{\frac{d}{4}}+\|\u\|_{\V}^{\frac{d}{4}}\right)\|\u_n-\u\|_{\H}^{1-\frac{d}{4}}\|\v\|_{\V}\nonumber\\&\to 0\ \text{ as }\ n\to\infty, 
\end{align}
where we have used the convergence \eqref{eqn-strong} to get the required result.
Thus, by using \eqref{2.62}, we obtain
\begin{align*}
	\|\B(\v_0,\u_n)-\B(\v_0,\u)\|_{\V^{\prime}+\L^{\frac{r+1}{r}}}=\sup_{\|\v\|_{\V\cap\L^{r+1}}\leq 1}|\langle\B(\v_0,\u_n-\u),\v\rangle| \to 0\ \text{ as }\ n\to\infty,
\end{align*}	
which implies $\B(\v_0,\u_n) \to \B(\v_0,\u)$ in $\V^{\prime}+\L^{\frac{r+1}{r}}$. Hence the operator $\B(\v_0,\cdot)$ is
completely continuous, which in turn, implies that $\B(\v_0,\cdot)$ is pseudomonotone.

As the components of $\mathcal{G}(\cdot)$, namely $\A$, $\B(\v_0,\cdot)$ and $\widehat{\mathcal{G}}(\cdot)$ are pseudomonotone and sum of pseudomonotone operators remains pseudomonotone (cf. \cite[Proposition 1.3.68]{ZdSmP}), we conclude that $\mathcal{G}(\cdot)$ is pseudomonotone.
	\vskip 0.2cm
\noindent
\textbf{Step III:} \emph{$\Phi(\cdot): \V\cap\L^{r+1} \to \mathbb{R}$ is convex and lower semicontinuous on $\V\cap\L^{r+1}$.} To establish the convexity of $\Phi$, let $\lambda \in [0,1]$ and take arbitrary $\u_1,\u_2 \in \V\cap\L^{r+1}$. Then
	\begin{align*}
	\Phi(\lambda \u_1+ (1-\lambda)\u_2)
	&= \int_{\Omega} g\|\mathbb{D}(\lambda \u_1+(1-\lambda)\u_2)\|_{\mathbb{M}^d}\d x \\
	&\leq \lambda \int_{\Omega} g\|\mathbb{D}\u_1\|_{\mathbb{M}^d}\d x+ (1-\lambda) \int_{\Omega} g\|\mathbb{D}\u_2\|_{\mathbb{M}^d}\d x \\
	&\leq \lambda \Phi(\u_1)+ (1-\lambda) \Phi(\u_2),
\end{align*}
which proves that $\Phi$ is convex. Now, by using H\"older's inequality and using the fact that $\V\cap\L^{r+1}\hookrightarrow	\V$, we obtain
	\begin{align}\label{lower-sem}
		|\Phi(\u_1)-\Phi(\u_2)|
		&= \int_{\Omega} |g|\bigl|\|\mathbb{D}\u_1\|_{\mathbb{M}^d}-\|\mathbb{D}\u_2\|_{\mathbb{M}^d}\bigr|\d x \nonumber \\
		&\leq \int_{\Omega} |g| \|\mathbb{D}(\u_1-\u_2)\|_{\mathbb{M}^d}\d x
		\leq \|g\|_{\mathrm{L}^2(\Omega)} \|\u_1-\u_2\|_{\V} \nonumber\\
		& \leq \|g\|_{\mathrm{L}^2(\Omega)} \|\u_1-\u_2\|_{\V\cap\L^{r+1}} \ \text{ for all }\  \u_1,\u_2 \in \V\cap\L^{r+1}.
	\end{align}
Consequently, $\Phi$ is Lipschitz continuous on $\V \cap \L^{r+1}$ and therefore it is lower semicontinuous. Furthermore, since $\Phi$ is convex, it follows from Proposition \ref{prop 2.9}  that $\Phi$ is weakly lower semicontinuous.
	\vskip 0.2cm
\noindent
\textbf{Step IV:} By using Lemma \ref{lemma-1}, the set-valued map $\mathbb{U}: \V\cap\L^{r+1} \multimap \V\cap\L^{r+1} $ defined in \eqref{2.20} has nonempty, closed and convex values.

Hence, by applying the \cite[Theorem 4]{MSSD2}, we conclude that \emph{Problem \ref{prob3.4}} admits a unique solution $\u \in \mathbb{U}(\v_0)$.
\end{proof}

Let us now show that the solution $\u$ of \emph{Problem \ref{prob3.4}} obtained in Theorem \ref{theorem3.5} is bounded. 
\begin{lemma}
 Assume that $\f\in \V^{\prime}$ and $\u$ be the solution of Problem \ref{prob3.4} obtained in Theorem \ref{theorem3.5}. Then $\u$ satisfies
 \begin{align}\label{2.67}
 	\|\u\|_{\V}^2+ \|\u\|_{\L^{r+1}}^{r+1} \leq 2\max\bigg\{\frac{1}{\gamma},\frac{1}{\beta}\bigg\}\widetilde{K},
 \end{align}
 where $\widetilde{K}= \frac{1}{2\gamma}\left(\|\f\|_{\V^{\prime}}+ \|\mathcal{L}^*\|\|\w_0\|_{\L^2(\Gamma_1)}+\|g\|_{\mathrm{L}^2(\Omega)}\right)^2+|\kappa|^{\frac{r+1}{r-q}} |\Omega|.$
\end{lemma}
\begin{proof}
In this proof, we establish that the solution $\u$
of \emph{Problem \ref{prob3.4}} satisfies a uniform boundedness estimate. For this, choose $\v = \boldsymbol{0} \in \mathbb{U}(\v_0)$ in \eqref{3.1} to get
 \begin{align*}
\langle \A \u + \B(\v_0,\u)+\alpha \u+\beta \C(\u)+\kappa \widetilde\C(\u)- \f+\mathcal{L}^*\w_0, -\u \rangle+\Phi(\boldsymbol{0})-\Phi(\u) \geq 0,
\end{align*}
which implies
\begin{align*}
&	\langle \A \u,\u \rangle +\langle\B(\v_0,\u),\u \rangle
+\langle\alpha \u+\beta \C(\u),\u \rangle	\nonumber\\&\leq \langle \f-\mathcal{L}^*\w_0,\u \rangle-\kappa \langle\widetilde\C(\u),\u \rangle
	+ \Phi(\boldsymbol{0})-\Phi(\u).
\end{align*}
By using the strong monotonicity of $\A$ (see \eqref{2.60}), the properties of $\B$ given in \eqref{eqn-b-est-1}, and the Lipschitz continuity of $\Phi$ from \eqref{lower-sem}, we obtain
\begin{align}\label{2.64}
\gamma \|\u\|_{\V}^2+\alpha \|\u\|_{\H}^2+\beta \|\u\|_{\L^{r+1}}^{r+1}
\leq \left(\|\f\|_{\V^{\prime}}+ \|\mathcal{L}^*\|\|\w_0\|_{\L^2(\Gamma_1)}+\|g\|_{\mathrm{L}^2(\Omega)}\right) \|\u\|_{\V}+|\kappa| \|\u\|_{\L^{q+1}}^{q+1}.
\end{align}
Using H\"older's and Young's inequalities, we find
\begin{align}\label{2.65}
	|\kappa| \|\u\|_{\L^{q+1}}^{q+1}
	&= |\kappa| \int_{\Omega} |\u(x)|^{q+1}\d x \leq |\kappa||\Omega|^{\frac{r-q}{r+1}}
	\left( \int_{\Omega} |\u(x)|^{r+1}\d x \right)^{\frac{q+1}{r+1}} \nonumber\\
	&= |\kappa| |\Omega|^{\frac{r-q}{r+1}} \|\u\|_{\L^{r+1}}^{q+1}\leq \frac{\beta}{2}\|\u\|_{\L^{r+1}}^{r+1}
	+ |\kappa|^{\frac{r+1}{r-q}} |\Omega|.
\end{align}
Substituting \eqref{2.65} into \eqref{2.64} and applying Young's inequality yield
%\begin{align*}
%&\frac{\gamma}{2} \|\u\|_{\V}^2+\alpha \|\u\|_{\H}^2+\frac{\beta}{2} \|\u\|_{\L^{r+1}}^{r+1}
%\nonumber\\&\leq \frac{1}{2 \gamma}\left(\|\f\|_{\V^{\prime}}+ \|\mathcal{L}^*\|\|\w_0\|_{\L^2(\Gamma_1)}+\|g\|_{\mathrm{L}^2(\Omega)}\right)^2+\frac{\gamma}{2} \|\u\|_{\V}^2+|\kappa|^{\frac{r+1}{r-q}} |\Omega|.
%\end{align*}
%Thus, we get 
\begin{align}\label{norm-v}
\frac{\gamma}{2} \|\u\|_{\V}^2+\frac{\beta}{2} \|\u\|_{\L^{r+1}}^{r+1}
	\leq \frac{1}{2\gamma}\left(\|\f\|_{\V^{\prime}}+ \|\mathcal{L}^*\|\|\w_0\|_{\L^2(\Gamma_1)}+\|g\|_{\mathrm{L}^2(\Omega)}\right)^2+|\kappa|^{\frac{r+1}{r-q}} |\Omega|,
\end{align}
which implies \eqref{2.67}. 
%\begin{align}\label{2.67}
%	\|\u\|_{\V}^2+ \|\u\|_{\L^{r+1}}^{r+1} \leq 2\max\bigg\{\frac{1}{\gamma},\frac{1}{\beta}\bigg\}\widetilde{K},
%\end{align}
%where $\widetilde{K}= \frac{1}{2\gamma}\left(\|\f\|_{\V^{\prime}}+ \|\mathcal{L}^*\|\|\w_0\|_{\L^2(\Gamma_1)}+\|g\|_{\mathrm{L}^2(\Omega)}\right)^2+|\kappa|^{\frac{r+1}{r-q}} |\Omega|.$
\end{proof}
%%%%%%%%%%%%%%%%%%%%%%%%%%%%%%%%%%%%%%%%%
%%%%%%%%%%%%%%%%%%%%%%%%%%%%%%%%%%%%%%%%%
%%%%%%%%%%%%%%%%%%%%%%%%%%%%%%%%%%%%%%%%%
Let us consider a map $\Pi: \V\cap\L^{r+1} \times \L^2(\Gamma_1) \to \V\cap\L^{r+1} $ defined by 
\begin{align}\label{def-p}
	\Pi(\v_0,\w_0)=\u, \quad \u \in \V\cap\L^{r+1},
\end{align}
 where $\u $ is the unique solution to \emph{Problem \ref{prob3.4}} corresponding to $(\v_0,\w_0) \in \V\cap\L^{r+1} \times \L^2(\Gamma_1).$ Next, we show that the map $\Pi: \V\cap\L^{r+1} \times \L^2(\Gamma_1) \to \V\cap\L^{r+1} $ is completely continuous.

\begin{lemma}\label{lemma-2}
Under the assumptions of Theorem \ref{theorem3.5}, the map $$\Pi: \V\cap\L^{r+1} \times \L^2(\Gamma_1) \to \V\cap\L^{r+1} $$ defined in \eqref{def-p} is completely continuous.
\end{lemma}
\begin{proof}
In order to prove the mapping $\Pi$ is completely continuous, that is,
continuous from $(\V\cap\L^{r+1})_w \times (\L^2(\Gamma_1))_w$ to $\V\cap\L^{r+1}$, let  $\{{\v_0}_n\}_{n \in \N} \subset \V\cap\L^{r+1}$, $\{{\w_0}_n\}_{n \in \N} \subset \L^2(\Gamma_1)$ be two sequences such that
$${\v_0}_n \xrightarrow{w}  \v_0 \ \text{ in }\  \V\cap\L^{r+1},~ {\w_0}_n \xrightarrow{w}  \w_0 \ \text{ in }\  \L^2(\Gamma_1),$$ 
and
$\u_n := \Pi({\v_0}_n,{\w_0}_n) \in \mathbb{U}({\v_0}_n)$. We have to prove that 
$$\u_n \to \u \ \text{ in }\  \V\cap\L^{r+1} \text{ and } \u =  \Pi(\v_0,\w_0) \in \mathbb{U}(\v_0).$$
The proof is carried out in two steps written as follows: 
\vskip 0.2cm
\noindent
\textbf{Step I:} \emph{Weak convergence.} Since $\u_n \in \V\cap\L^{r+1}$ and $\u_n \in \mathbb{U}({\v_0}_n)$, thus it follows from \emph{Problem \ref{prob3.4}} that for every $\v \in \mathbb{U}({\v_0}_n)$, we have 
\begin{equation}\label{3.7}
	\langle \widetilde{F}(\u_n)+\B({\v_0}_n,\u_n)- \h_n, \v-\u_n \rangle+\Phi(\v)-\Phi(\u_n)  \geq 0,
\end{equation}
where $\widetilde{F}(\u_n):=\A \u_n+ \widehat{\mathcal{G}}(\u_n)=\A \u_n+\alpha \u_n+\beta \C(\u_n)+\kappa \widetilde\C(\u_n) \text{ and }\h_n := \f - \mathcal{L}^*{\w_0}_n$.

Since, every weakly convergent sequence is bounded, the sequence $\{{\w_0}_n\}_{n\in\N}$ is bounded independent of $n$ in $\L^2(\Gamma_1)$, thus from \eqref{2.67}, it follows that $\{\u_n\}_{n\in\N}$ is bounded in $\V\cap\L^{r+1}$ and using the reflexivity of $\V\cap\L^{r+1}$, there exists a weakly convergent subsequence (still denoted by the same symbol) such that 
\begin{align}\label{weak-u}
	\u_n \xrightarrow{w} \u \ \text{ in }\  \V\cap\L^{r+1},
\end{align} 
which implies 
\begin{align*}
	\u_n \xrightarrow{w}\u \ \ \text{ in }\ \ \V\ \text{ and }\ \u_n \xrightarrow{w}\u\ \ \text{ in }\ \ \L^{r+1}.
\end{align*}
From the convergence ${\v_0}_n \xrightarrow{w}  \v_0 \ \text{ in }\  \V\cap\L^{r+1}$, we get 
\begin{align*}
	{\v_0}_n \xrightarrow{w}\v_0 \ \ \text{ in }\ \ \V\ \text{ and }\ {\v_0}_n \xrightarrow{w}\v_0\ \ \text{ in }\ \ \L^{r+1}.
\end{align*}
By using the fact that the embedding $\V \hookrightarrow\H$ is compact, there exists subsequences of $\{{\u}_n\}_{n\in\N}$ and $\{{\v_0}_n\}_{n\in\N}$ (still denoted by the same symbol) such that 
\begin{align}\label{eqn-strong1}
	\u_n \to \u \ \ \text{ in }\ \ \H\  \text{ and } \ {\v_0}_n \to \v_0 \ \ \text{ in }\ \  \H. 
\end{align}
Combining the facts $\u_n \in \mathbb{U}({\v_0}_n)$ along with \eqref{weak-u}, it follows from Lemma \ref{lemma-1} that $\u \in \mathbb{U}(\v_0)$.

Let us now consider $\v \in \mathbb{U}(\v_0)$. By Lemma \ref{lemma-1}, we employ condition (M$_2$) in Definition \ref{def:mosco} for $\v,~\u \in \mathbb{U}(\v_0)$ and thus obtain two sequences
$\{\v_n\}_{n \in \N}$ and $\{\boldsymbol{\eta}_n\}_{n \in \N}$ such that
\begin{align}\label{z-con}
	\v_n, \boldsymbol{\eta}_n \in \mathbb{U}({\v_0}_n),\
	\v_n \to \v\  \text{ and }\  \boldsymbol{\eta}_n \to \u \ \text{in } \V\cap\L^{r+1} \ \text{ as }\ n \to \infty.
\end{align}
Since $\boldsymbol{\eta}_n \in \mathbb{U}({\v_0}_n)$,  we choose $\v = \boldsymbol{\eta}_n$  in \eqref{3.7} to get
\begin{equation}\label{2.100}
\langle \widetilde{F}(\u_n)+\B({\v_0}_n,\u_n), \u_n -\boldsymbol{\eta}_n \rangle \leq \langle \h_n, \u_n - \boldsymbol{\eta}_n \rangle
	+ \Phi(\boldsymbol{\eta}_n) - \Phi(\u_n).
\end{equation}
We calculate $	|\langle\B({\v_0}_n ,\u_n)-\B(\v_0 ,\u),\z\rangle|$ as 
\begin{align}\label{2.731}
	|\langle\B({\v_0}_n ,\u_n)-\B(\v_0 ,\u),\z\rangle|
	\leq |\langle\B({\v_0}_n,\u_n-\u),\z\rangle|+|\langle\B({\v_0}_n-\v_0,\u),\z\rangle|.
\end{align}
By using H\"older's, Gagliardo-Nirenberg's inequalities and \eqref{eqn-strong1}, we estimate
\begin{align}\label{estimate-B-1}
	|\langle\B({\v_0}_n-\v_0,\u),\z\rangle|
	&\leq \|{\v_0}_n-\v_0\|_{\L^4}\|\nabla \u\|_{\H}\|\z\|_{\L^4}\nonumber\\&\leq C \|\z\|_{\L^4}\|{\v_0}_n-\v_0\|_{\H}^{1-\frac{d}{4}}\|{\v_0}_n-\v_0\|_{\H^1}^{\frac{d}{4}}\|\u\|_{\V}\nonumber\\&\leq C\|\z\|_{\V}\left(\|{\v_0}_n\|_{\V}^{\frac{d}{4}}+\|\v_0\|_{\V}^{\frac{d}{4}}\right)\|{\v_0}_n-\v_0\|_{\H}^{1-\frac{d}{4}}\|\u\|_{\V}\nonumber\\&\to 0\ \text{ as }\ n\to\infty.
\end{align}
A calculation similar to \eqref{estimate-B-1} and then using \eqref{eqn-strong1} yield
\begin{align}\label{estimate-B-2}
	|\langle\B({\v_0}_n,\u_n-\u),\z\rangle| \to 0\ \text{ as }\ n\to\infty.
\end{align}
Using \eqref{estimate-B-1} and \eqref{estimate-B-2} in \eqref{2.731}, we obtain
\begin{align*}
	\|\B({\v_0}_n ,\u_n)-\B(\v_0 ,\u)\|_{\V^{\prime}}=\sup_{\|\z\|_{\V}\leq 1}|\langle\B({\v_0}_n ,\u_n)-\B(\v_0 ,\u),\z\rangle| \to 0,
\end{align*}
which implies 
\begin{align}\label{strong-con-1}
	\B({\v_0}_n ,\u_n) \to \B(\v_0 ,\u) \ \text{ in }\  \V^{\prime}.
\end{align}
%\begin{align}
%	|\langle\B({\v_0}_n,\u_n-\u),\v\rangle|
%	&\leq \|\v_0\|_{\L^4}\|\nabla \v\|_{\H}\|\u_n-\u\|_{\L^4}\nonumber\\&\leq C \|\v_0\|_{\L^4}\|\u_n-\u\|_{\H}^{1-\frac{d}{4}}\|\u_n-\u\|_{\H^1}^{\frac{d}{4}}\|\v\|_{\V}\nonumber\\&\leq C\|\v_0\|_{\V}\left(\|\u_n\|_{\V}^{\frac{d}{4}}+\|\u\|_{\V}^{\frac{d}{4}}\right)\|\u_n-\u\|_{\H}^{1-\frac{d}{4}}\|\v\|_{\V}\nonumber\\&\to 0\ \text{ as }\ n\to\infty,
%\end{align}
According to the trace theorem \cite[Theorem 2.21]{SMAOMS}, $\mathcal{L}: \V \to \L^2(\Gamma_1)$ is a continuous, linear and compact operator, thus we have 
$$ \f- \mathcal{L}^* {\w_0}_n \to \f -\mathcal{L}^*\w_0 \ \text{ in }\  \V^{\prime},$$
which implies
\begin{align}\label{2.78}
	\h_n \to \h \ \text{ in }\  \V^{\prime}.
\end{align}
Moreover, continuity of $\Phi$ (see \eqref{lower-sem}) along with \eqref{weak-u} and \eqref{z-con} imply
\begin{align}\label{convergences}
\Phi(\u_n) \xrightarrow{w} \Phi(\u)	\ \text{ and } \  \Phi(\boldsymbol{\eta}_n) \to \Phi(\u) \ \text{ as }\ n \to \infty, \text{ respectively}.
\end{align}
By using the convergences given in \eqref{convergences}, we get 
\begin{align}\label{2.80}
	\limsup_{n \to \infty}(\Phi(\boldsymbol{\eta}_n)-\Phi(\u_n))\leq \Phi(\u)-\liminf_{n \to \infty}\Phi(\u_n) \leq \Phi(\u)-\Phi(\u)=0.
\end{align}
Thus, by using \eqref{2.100} along with \eqref{strong-con-1}, we obtain
\begin{align}\label{2.83}
&\limsup_{n \to \infty} \langle \widetilde{F}(\u_n), \u_n - \u \rangle
\nonumber\\	&= \limsup_{n \to \infty} \langle \widetilde{F}(\u_n)+\B({\v_0}_n,\u_n),\u_n - \boldsymbol{\eta}_n \rangle \nonumber \\
	&\quad +\limsup_{n \to \infty} \langle \widetilde{F}(\u_n) +\B({\v_0}_n,\u_n),\boldsymbol{\eta}_n-\u \rangle+ \lim_{n \to \infty}  \langle \B({\v_0}_n,\u_n), \u - \u_n \rangle  \nonumber \\
	&\leq \limsup_{n \to \infty}
	\left( \langle \h_n, \u_n - \boldsymbol{\eta}_n\rangle
	+ \Phi(\boldsymbol{\eta}_n) - \Phi(\u_n)\right)  \nonumber  \\
	&\quad+\limsup_{n \to \infty} \langle \widetilde{F}(\u_n) +\B({\v_0}_n,\u_n),\boldsymbol{\eta}_n-\u \rangle+ \lim_{n \to \infty}  \langle \B({\v_0}_n,\u_n), \u - \u_n \rangle.
\end{align}
By using \eqref{2.78} along with the fact that $\u_n-\boldsymbol{\eta}_n \xrightarrow{w} \boldsymbol{0} \ \text{ in }\  \V \cap \L^{r+1}$ and \eqref{2.80}, the first term of \eqref{2.83} gives 
\begin{align}\label{First-term}
	\limsup_{n \to \infty}
	\left( \langle \h_n, \u_n - \boldsymbol{\eta}_n\rangle
	+ \Phi(\boldsymbol{\eta}_n) - \Phi(\u_n)\right)  \leq 0. 
\end{align}
As the components of $\widetilde{F}(\cdot)$, namely $\A$ and $\widehat{\mathcal{G}}(\cdot)$ are pseudomonotone (see \textbf{Step II} of Theorem \ref{theorem3.5}) and sum of pseudomonotone operators remains pseudomonotone, it follows $\widetilde{F}(\cdot)$ is a pseudomonotone operator. Hence, by using \eqref{eqn-con-pseudo}, we get 
\begin{align}\label{F-con}
\widetilde{F}(\u_n)\xrightarrow{w}\widetilde{F}(\u)\ \ \text{ in }\  \ \V^{\prime}+\L^{\frac{r+1}{r}}.
		\end{align}
Thus, by using \eqref{z-con}, \eqref{strong-con-1} and \eqref{F-con},  the second term in the right hand side of \eqref{2.83} gives
 \begin{align}\label{Second-term}
 	\langle \widetilde{F}(\u_n) +\B({\v_0}_n,\u_n),\boldsymbol{\eta}_n-\u \rangle \to 0 \ \text{ as }\  n\to\infty.
 \end{align}
By using \eqref{weak-u} and \eqref{strong-con-1}, the third term in the right hand side of \eqref{2.83} gives
\begin{align}\label{Third-term}
	 \lim_{n \to \infty}  \langle \B({\v_0}_n,\u_n), \u - \u_n \rangle=0.
\end{align}
Hence, using the convergences \eqref{First-term},  \eqref{Second-term} and  \eqref{Third-term} in \eqref{2.83}, we obtain
%Here, we have used(3.7) and the convergences un − ηn 0inV, ηn −u0 →0inV,and
%B(vn,un) →B(v,u0) in V∗. 
\begin{align}\label{2.831}
	\limsup_{n \to \infty} \langle \widetilde{F}(\u_n), \u_n - \u \rangle \leq 0.
\end{align}
By combining the convergence \eqref{weak-u} with \eqref{2.831} and the pseudomonotonicity of the operator $\widetilde{F}(\cdot)$ yields
\begin{align}\label{2.75}
	\langle \widetilde{F}(\u), \u- \v \rangle
	&\leq \liminf_{n \to \infty} \langle \widetilde{F}(\u_n), \u_n - \v \rangle \leq \limsup_{n \to \infty} \langle \widetilde{F}(\u_n), \u_n - \v \rangle \nonumber \\
	&= \limsup_{n \to \infty} \langle \widetilde{F}(\u_n), \u_n - \v_n \rangle
	+ \lim_{n \to \infty} \langle \widetilde{F}(\u_n), \v_n - \v \rangle \nonumber\\
	&= \limsup_{n \to \infty} \langle \widetilde{F}(\u_n),\u_n - \v_n \rangle,
\end{align}
where we have used the fact that
\begin{align*}
	\lim_{n \to \infty} \langle \widetilde{F}(\u_n), \v_n - \v \rangle = 0,
\end{align*} 
which is obtained by using boundedness of the operator $\widetilde{F}$ (see \textbf{Step II} of Theorem \ref{theorem3.5}) together with the strong convergence \eqref{z-con}.
Also, $\v_n \in \mathbb{U}({\v_0}_n)$, thus in particular, putting $\v = \v_n $ in \eqref{3.7}, we get
\begin{align}\label{2.73}
\langle \widetilde{F}(\u_n), \u_n-\v_n \rangle \leq \langle \B({\v_0}_n,\u_n)-\h_n, \v_n-\u_n\rangle +\Phi(\v_n)-\Phi(\u_n).
\end{align}
Again, by doing similar calculations to \eqref{2.80}, we obtain
\begin{align}\label{2.79}
\limsup_{n \to \infty}
\left(\Phi(\v_n)-\Phi(\u_n) \right)
\leq \Phi(\v)-\Phi(\u)\  \ \text{ as }\ \ n \to \infty.
\end{align}
Thus, by combining \eqref{2.75} and \eqref{2.73}, and using \eqref{strong-con-1}, \eqref{2.78}, \eqref{2.79}, we find 
\begin{align*}
	\langle \widetilde{F}(\u), \u- \v \rangle
	&\leq  \lim_{n \to \infty} \langle \B({\v_0}_n,\u_n)-\h_n, \v_n-\u_n\rangle 
	+ \limsup_{n \to \infty}
	\left(\Phi(\v_n)-\Phi(\u_n) \right) \\
	&\leq \langle \B(\v_0,\u)-\h, \v - \u \rangle+\Phi(\v)-\Phi(\u),
\end{align*}
so that 
$$\langle \widetilde{F}(\u)+\B(\v_0,\u)-\h,\v-\u\rangle +\Phi(\v)-\Phi(\u)\geq 0.$$
Since $\v \in \mathbb{U}(\v_0)$ is arbitrary, we conclude that $\u \in \mathbb{U}(\v_0)$ is a solution to the limit problem
associated with \eqref{3.7}, that is, $\u=  \Pi(\v_0,\w_0)$.  

%Moreover, the uniqueness of the limit element $\u$ implies that the whole sequence $\{\u_n\}$ convergences weakly to $u_0$ in $\V\cap\L^{r+1}$.
\vskip 0.2cm
\noindent
\textbf{Step II:} \emph{Strong convergence.} Finally, we show the strong convergence of the sequence $\{\u_n\}_{n \in \N}$, that is,
$$ \u_n \to \u \ \text{ in }\  \V\cap\L^{r+1}.$$
Since $\u \in \mathbb{U}(\v_0)$, then by employing condition (M$_2$) in Definition \ref{def:mosco}, we obtain a sequence $\{\boldsymbol{\eta}_n\}_{n \in \N} \in \mathbb{U}({\v_0}_n)$ such that
\begin{align*}
\boldsymbol{\eta}_n \to \u \ \text{ in } \ \V\cap\L^{r+1} \ \text{ as }\ n \to \infty.
\end{align*}
By using \eqref{2.100}, we calculate
\begin{align}\label{3.15}
&\limsup_{n \to \infty} \langle \widetilde{F}(\u_n) + \B({\v_0}_n,\u_n),\u_n - \u \rangle
\nonumber\\	&\le \limsup_{n \to \infty} \langle \widetilde{F}(\u_n) + \B({\v_0}_n,\u_n), \u_n - \boldsymbol{\eta}_n \rangle \nonumber \\
	&\quad + \limsup_{n \to \infty} \langle \widetilde{F}(\u_n) + \B({\v_0}_n,\u_n), \boldsymbol{\eta}_n - \u \rangle \nonumber  \\
	&\leq \limsup_{n \to \infty}
	\left( \langle \h_n, \u_n - \boldsymbol{\eta}_n\rangle
	+ \Phi(\boldsymbol{\eta}_n) - \Phi(\u_n)\right)  \nonumber  \\
	& \quad +\limsup_{n \to \infty} \langle \widetilde{F}(\u_n) +\B({\v_0}_n,\u_n),\boldsymbol{\eta}_n-\u \rangle \leq 0,
\end{align}
where we have employed \eqref{2.78} together with the fact that $\u_n-\boldsymbol{\eta}_n \xrightarrow{w} \boldsymbol{0} \ \text{ in }\  \V \cap \L^{r+1}$, as well as \eqref{2.80} and \eqref{Second-term}. From \eqref{estimate-B-1}, it follows that
\begin{align}\label{Converg-1}
	\B({\v_0}_n,\u) \to \B(\v_0,\u) \ \text{ in }\  \V^{\prime}.
\end{align}  
Consequently,
\begin{align}\label{strong-con}
	&\langle \widetilde{F}(\u)+ \B({\v_0}_n,\u), \u_n  - \u \rangle \nonumber \\ &=\langle \widetilde{F}(\u)+\B(\v_0,\u),\u_n-\u\rangle + \langle \B({\v_0}_n,\u)-\B(\v_0,\u),\u_n-\u\rangle \to 0,
\end{align}
where the first term tends to zero in view of \eqref{weak-u} and the second term converges to zero by the strong convergence \eqref{Converg-1} together with \eqref{weak-u}.

Since, $\widetilde{F}(\cdot) + \B({\v_0}_n,\cdot)$ is strongly monotone for $\frac{\gamma (1-\varepsilon)}{C^2_pC^2_k}+\alpha \geq 2 \varrho$ (see \eqref{4.13}) and by using \eqref{3.15} and \eqref{strong-con}, it follows that
\begin{align*}
&\omega \limsup_{n \to \infty} \left(\|\u_n - \u\|_{\V}^2+\|\u_n - \u
\|_{\L^{r+1}}^{r+1}\right)  \\ 
&\leq \limsup_{n \to \infty} \langle \widetilde{F}(\u_n)+ \B({\v_0}_n,\u_n)-  (\widetilde{F}(\u)+ \B({\v_0}_n,\u)), \u_n  - \u \rangle  \\ 
& \leq \limsup_{n \to \infty} \langle \widetilde{F}(\u_n)+ \B({\v_0}_n,\u_n), \u_n  - \u \rangle-\liminf_{n \to \infty}\langle \widetilde{F}(\u)+ \B({\v_0}_n,\u), \u_n  - \u \rangle \leq 0.
\end{align*}
Therefore, we have 
$$\|\u_n - \u\|_{\V} \to 0 \text{ and } \|\u_n - \u\|_{\L^{r+1}} \to 0 \ \text{ as }\ n \to \infty,$$
which further implies $\u_n \to \u \text { in } \V \cap \L^{r+1} \ \text{ as }\ n \to \infty.$
\end{proof}
Next we consider the following problem:
\begin{problem}\label{Problem 3.1}
	Find $\u \in \V\cap\L^{r+1}$ such that $\u \in \mathbb{U}(\u)$ and there is $\w_0 \in \partial J(\mathcal{L}\u,\mathcal{L}\u)$ with
	\begin{align*}
		\langle\mathcal{F}(\u)- \f,\v-\u\rangle+\langle\w_0, \mathcal{L}\v-\mathcal{L}\u \rangle+\Phi(\v)-\Phi(\u)  \geq 0 \ \text{ for all }\  \v \in \mathbb{U}(\u).
	\end{align*}
\end{problem}
Let us now introduce a bounded set on $\V\cap\L^{r+1} \times \L^2(\Gamma_1)$ by
\begin{align}\label{def-1}
	\mathcal{D}:= \{(\v_0,\w_0)\in \V\cap\L^{r+1} \times \L^2(\Gamma_1)~|~ \|\v_0\|_{ \V\cap\L^{r+1}} \leq s_1, \|\w_0\|_{\L^2(\Gamma_1)} \leq s_2\},
\end{align}
where 
\begin{equation}\label{3.36}
	s_1:=
	\begin{cases}
		\frac{C_0 \gamma + C_1\|\mathcal{L}^*\|}{\gamma - (C_2+C_3) \|\mathcal{L}\|\|\mathcal{L}^*\|}, \ 0 < (C_2+C_3)\|\mathcal{L}\|\,\|\mathcal{L}^*\|< \gamma
		& \text{for } r\in\left[1,\dfrac{d+2}{d-2}\right],
		\\ 
		2\Bigg\{
		\sqrt{\frac{2 |\kappa|^{\frac{r+1}{r-q}} |\Omega|}{\gamma}}+
	\frac{\sqrt{2}\left(\|\f\|_{\V^{\prime}}+\|g\|_{\mathrm{L}^2(\Omega)} \right)}{\gamma}+\frac{\sqrt{2}\|\mathcal{L}^*\|C_1}{\gamma}+ \left(\frac{2 \|\mathcal{L}^*\|^2C_1^2}{\beta \gamma}\right)^{\frac{1}{r+1}}
		\\ 
		 \qquad
	+ \left(\frac{2}{\beta}\right)^{\frac{1}{r+1}}
	\left(|\kappa|^{\frac{1}{r-q}} |\Omega|^{\frac{1}{r+1}}
	+\frac{\left(\|\f\|_{\V^{\prime}}+\|g\|_{\mathrm{L}^2(\Omega)}\right)^{\frac{2}{r+1}}}{\gamma^{\frac{1}{r+1}}}\right)
		\Bigg\},
		& \text{for } r>\dfrac{d+2}{d-2},
	\end{cases}
\end{equation}
and 
\begin{align}\label{eq-s2}
	s_2 := C_1+(C_2+C_3)\|\mathcal{L}\|s_1.
\end{align} 
Here, $s_1$ is chosen according to the corresponding range of $r$ specified in \eqref{3.36}, and $s_2$ is determined accordingly. Moreover,
\begin{align}\label{Value-c}
	C_0 := \frac{1}{\gamma}\left(\|\f\|_{\V^{\prime}}+\|g\|_{\mathrm{L}^2(\Omega)}\right)+\bigg(\frac{2}{\gamma}|\kappa|^{\frac{r+1}{r-q}} |\Omega|\bigg)^{\frac{1}{2}}> 0,
\end{align}
and the constants $C_i$'s for $i=1,2,3$ are defined in Lemma \ref{lem-J-lem}. Note that the condition $r\in\left[1,\dfrac{d+2}{d-2}\right],$ means that $r\in[1,\infty)$ for $d=1,2$ and $r\in\left[1,\dfrac{d+2}{d-2}\right],$ for $d\geq 3$. 
\begin{theorem}\label{theorem-ex}
Assume that all hypotheses of Theorem \ref{theorem3.5} and Lemma \ref{lem-J-lem} are satisfied. Then under the following smallness condition
\begin{align}
0 &< (C_2+C_3)\|\mathcal{L}\|\,\|\mathcal{L}^*\|< \gamma,
	\ \text{ for }\ r\in\left[1,\frac{d+2}{d-2}\right],
	\label{small-1}
	\\[2mm]
	0 &<
	\frac{\sqrt{2}\,\|\mathcal{L}^*\|(C_2+C_3)\|\mathcal{L}\|}{\gamma}+
	\frac{
		\bigl(\sqrt{2}\,\|\mathcal{L}^*\|(C_2+C_3)\|\mathcal{L}\|\bigr)^{\frac{2}{r+1}}
		s_1^{\frac{1-r}{r+1}}
	}{(\beta\gamma)^{\frac{1}{r+1}}}
	\le \frac{1}{2},\  \text{ for } \ r>\frac{d+2}{d-2},\label{small-2}
\end{align}
there exists a solution $\u \in \V\cap\L^{r+1}$ to Problem \ref{Problem 3.1}.
\end{theorem}
\begin{proof} 
Let $\v_0 \in \V\cap\L^{r+1}$ and $\w_0 \in \L^2(\Gamma_1)$ be fixed and recall that $\Pi : \V\cap\L^{r+1} \times \L^2(\Gamma_1) \to \V\cap\L^{r+1}$ denote the solution map for \emph{Problem \ref{prob3.4}} (see Theorem \ref{theorem3.5}). Now, consider the set-valued map $\Xi : \mathcal{D} \multimap \mathcal{D}$ defined by
		\begin{align*}
			\Xi(\v_0,\w_0) := \left(\Pi(\v_0,\w_0),H(\mathcal{L}\Pi(\v_0,\w_0)\right) =  \left(\u, H(\mathcal{L}\u) \right)\  \text{ for }\  (\v_0,\w_0) \in \mathcal{D},
		\end{align*}
where $H : \L^2(\Gamma_1) \multimap \L^2(\Gamma_1)$ is defined by 
\begin{align*}
	H(\tilde{\z}):= \partial J(\tilde{\z},\tilde{\z}) \ \text{ for } \ \tilde{\z} \in \L^2(\Gamma_1).
\end{align*} 
%and 
%\begin{align}\label{def-1}
%\mathcal{D}:= \{(\v_0,\w_0)\in \V\cap\L^{r+1} \times \L^2(\Gamma_1)~|~ \|\v_0\|_{ \V\cap\L^{r+1}} \leq s_1, \|\w_0\|_{\L^2(\Gamma_1)} \leq s_2\}
%\end{align}
%for some $s_1, s_2 > 0. $ 
We verify that $\Xi$ satisfies all the assumptions of Theorem \ref{theorem 2.10}, proceeding step by step as follows:
\vskip 0.2cm
\noindent
\textbf{Step I:} Firstly, we show that, for the positive constants \(s_1\) and \(s_2\) defining the set \(\mathcal{D}\) in \eqref{def-1}, the inclusion
$$
\Xi(\v_0,\w_0)\subset \mathcal{D}
\quad \text{for all } (\v_0,\w_0)\in\mathcal{D}$$
holds. To establish this, we distinguish the following cases according to the values of \(r\).
\vskip 0.2cm
\noindent
\textbf{Case I:} \emph{$d= 2$ with $1 \leq r < \infty$ and $d= 3$ with $1 \leq r \leq \frac{d+2}{d-2}$.} By Sobolev's embedding, we have $ \V \hookrightarrow \L^{r+1},$
for $ 2 \leq r+1 \leq \frac{2d}{d-2} (1 \leq r < \infty,\text{ for } d=1,2),$
which guarantees that the $\L^{r+1}$-norm is controlled by the $\V$-norm. As a result, the $\|\u\|_{\V \cap \L^{r+1}}$ can be equivalently defined in terms of  $\|\u\|_{\V}$. 
%For proving this, consider
%\begin{align}
%	s_1 &:= \frac{C_0 \gamma + C_1\|\mathcal{L}^*\|}{\gamma - (C_2+C_3) \|\mathcal{L}\|\|\mathcal{L}^*\|}, \label{eq-s1}\\
%	s_2 &:= C_1+(C_2+C_3)\|\mathcal{L}\|s_1, \label{eq-s2}
%\end{align}
%Since, the value of $s_1$ given in \eqref{small-1} is well-defined by using the condition \eqref{small-1}. 
%Thus,
%where, the value of $C_0$ is obtained from \eqref{norm-v}.
Utilizing the fact
\begin{align}\label{inequaity-1}
	(\varsigma_1 + \varsigma_2)^{\Theta} \leq  \varsigma_1^{\Theta}+\varsigma_1^{\Theta} \ \text{ for all }\  \varsigma_1, \varsigma_2 \geq 0 \text{ and } \Theta \in (0, 1),
\end{align}
and combining it with the estimate \eqref{norm-v}, we obtain the following bound:
\begin{align}\label{norm-v-1}
\|\u\|_{\V} \leq \frac{1}{\gamma}\left(\|\f\|_{\V^{\prime}}+ \|\mathcal{L}^*\|\|\w_0\|_{\L^2(\Gamma_1)}+\|g\|_{\mathrm{L}^2(\Omega)}\right)
+\bigg(\frac{2}{\gamma}|\kappa|^{\frac{r+1}{r-q}} |\Omega|\bigg)^{\frac{1}{2}}.
\end{align}
By substituting the value of the constant $C_0$ from \eqref{Value-c} and using definition \eqref{def-1}, the above inequality \eqref{norm-v-1} can be rewritten as
\begin{align}\label{3.90}
	\|\u\|_{\V} \leq  C_0+\frac{1}{\gamma}\|\mathcal{L}^*\|s_2.
\end{align}
Using the values of $s_1 \text{ and } s_2 $ from \eqref{3.36} and \eqref{eq-s2} in \eqref{3.90}, we obtain 
\begin{align*}
		\|\u\|_{\V} \leq C_0+\frac{1}{\gamma}\|\mathcal{L}^*\|\bigg(C_1+\frac{(C_2+C_3)\|\mathcal{L}\|(C_0 \gamma+ C_1\|\mathcal{L}^*\|)}{\gamma - (C_2+C_3) \|\mathcal{L}\|\|\mathcal{L}^*\|}\bigg)=s_1,
\end{align*}
which implies $\|\u\|_{\V \cap \L^{r+1}} \leq s_1.$ 
\vskip 0.2cm
\noindent
\textbf{Case II:} \emph{$d = 3$  with $r \in (5,\infty)$.} For $d = 3$ with $r \in (5,\infty)$, by using estimate \eqref{norm-v}, we   compute in the following way:
\begin{align*}
\|\u\|_{\V \cap \L^{r+1}}&= \sqrt{ \|\u\|_{\V}^{2} + \|\u\|_{\L^{r+1}}^{2} }  \\
&\leq \left\{\frac{2}{\gamma}\left[
|\kappa|^{\frac{r+1}{r-q}} |\Omega|
+ \frac{\left( \|\f\|_{\V^{\prime}}+\|g\|_{\mathrm{L}^2(\Omega)} \right)^{2}}{\gamma}
+ \frac{\|\mathcal{L}^*\|^2\|\w_0\|_{\L^2(\Gamma_1)}^2}{\gamma}\right]
	\right.\\
&\quad
\left.+\left[
\frac{2}{\beta}	\left(|\kappa|^{\frac{r+1}{r-q}} |\Omega|
	+ \frac{\left( \|\f\|_{\V^{\prime}}+\|g\|_{\mathrm{L}^2(\Omega)} \right)^{2}}{\gamma}
	+ \frac{\|\mathcal{L}^*\|^2\|\w_0\|_{\L^2(\Gamma_1)}^2}{\gamma}\right)
	\right]^{\frac{2}{r+1}}
	\right\}^{\frac{1}{2}}.
\end{align*}
By using the \eqref{inequaity-1} and \eqref{eq-s2}, we deduce
\begin{align*}
&\|\u\|_{\V \cap \L^{r+1}} \\
&\leq 
		\sqrt{\frac{2 |\kappa|^{\frac{r+1}{r-q}} |\Omega|}{\gamma}}+
\frac{\sqrt{2}\left(\|\f\|_{\V^{\prime}}+\|g\|_{\mathrm{L}^2(\Omega)} \right)}{\gamma} + \left(\frac{2}{\beta}\right)^{\frac{1}{r+1}}
\left(|\kappa|^{\frac{1}{r-q}} |\Omega|^{\frac{1}{r+1}}
	+\frac{\left(\|\f\|_{\V^{\prime}}+\|g\|_{\mathrm{L}^2(\Omega)}\right)^{\frac{2}{r+1}}}{\gamma^{\frac{1}{r+1}}}\right)
\\&\quad
		+\frac{\sqrt{2}\|\mathcal{L}^*\|\|\w_0\|_{\L^2(\Gamma_1)}}{\gamma}
		+ \left(\frac{2}{\beta}\right)^{\frac{1}{r+1}}
		\frac{\|\mathcal{L}^*\|^{\frac{2}{r+1}} \|\w_0\|_{\L^2(\Gamma_1)}^{\frac{2}{r+1}}}
		{\gamma^{\frac{1}{r+1}}}	\\
&\leq 	\sqrt{\frac{2 |\kappa|^{\frac{r+1}{r-q}} |\Omega|}{\gamma}}+
\frac{\sqrt{2}\left(\|\f\|_{\V^{\prime}}+\|g\|_{\mathrm{L}^2(\Omega)} \right)}{\gamma} + \left(\frac{2}{\beta}\right)^{\frac{1}{r+1}}
\left(|\kappa|^{\frac{1}{r-q}} |\Omega|^{\frac{1}{r+1}}
+\frac{\left(\|\f\|_{\V^{\prime}}+\|g\|_{\mathrm{L}^2(\Omega)}\right)^{\frac{2}{r+1}}}{\gamma^{\frac{1}{r+1}}}\right) \\&\quad
+\frac{\sqrt{2}\|\mathcal{L}^*\|}{\gamma}(C_1+(C_2+C_3)\|\mathcal{L}\|s_1)
+ \left(\frac{2 \|\mathcal{L}^*\|^2}{\beta \gamma}\right)^{\frac{1}{r+1}}
 (C_1+(C_2+C_3)\|\mathcal{L}\|s_1)^{\frac{2}{r+1}}\\
&\leq \sqrt{\frac{2 |\kappa|^{\frac{r+1}{r-q}} |\Omega|}{\gamma}}+
\frac{\sqrt{2}\left(\|\f\|_{\V^{\prime}}+\|g\|_{\mathrm{L}^2(\Omega)} \right)}{\gamma} + \left(\frac{2}{\beta}\right)^{\frac{1}{r+1}}
\left(|\kappa|^{\frac{1}{r-q}} |\Omega|^{\frac{1}{r+1}}
+\frac{\left(\|\f\|_{\V^{\prime}}+\|g\|_{\mathrm{L}^2(\Omega)}\right)^{\frac{2}{r+1}}}{\gamma^{\frac{1}{r+1}}}\right)\\
&\quad+\frac{\sqrt{2}\|\mathcal{L}^*\|C_1}{\gamma}+ \left(\frac{2 \|\mathcal{L}^*\|^2C_1^2}{\beta \gamma}\right)^{\frac{1}{r+1}}+\frac{\sqrt{2}\|\mathcal{L}^*\|(C_2+C_3)\|\mathcal{L}\|s_1}{\gamma}
+\frac{(\sqrt{2} \|\mathcal{L}^*\|(C_2+C_3)\|\mathcal{L}\|s_1)^\frac{2}{r+1}}{(\beta \gamma)^\frac{1}{r+1}} \\
&\leq \sqrt{\frac{2 |\kappa|^{\frac{r+1}{r-q}} |\Omega|}{\gamma}}+
\frac{\sqrt{2}\left(\|\f\|_{\V^{\prime}}+\|g\|_{\mathrm{L}^2(\Omega)} \right)}{\gamma} + \left(\frac{2}{\beta}\right)^{\frac{1}{r+1}}
\left(|\kappa|^{\frac{1}{r-q}} |\Omega|^{\frac{1}{r+1}}
+\frac{\left(\|\f\|_{\V^{\prime}}+\|g\|_{\mathrm{L}^2(\Omega)}\right)^{\frac{2}{r+1}}}{\gamma^{\frac{1}{r+1}}}\right)\\
&\quad+\frac{\sqrt{2}\|\mathcal{L}^*\|C_1}{\gamma}+ \left(\frac{2 \|\mathcal{L}^*\|^2C_1^2}{\beta \gamma}\right)^{\frac{1}{r+1}}\nonumber\\&\quad+\left(\frac{\sqrt{2}\|\mathcal{L}^*\|(C_2+C_3)\|\mathcal{L}\|}{\gamma}+\frac{(\sqrt{2} \|\mathcal{L}^*\|(C_2+C_3)\|\mathcal{L}\|)^\frac{2}{r+1}s_1^{\frac{1-r}{r+1}}}{(\beta \gamma)^\frac{1}{r+1}}\right)s_1\\
&\leq \frac{s_1}{2}+\frac{s_1}{2}=s_1,
\end{align*}
where we have used \eqref{3.36} and the smallness condition \eqref{small-2}. Further, by using Lemma \ref{lem-J-lem}, we have
\begin{align*}
	\|H(\mathcal{L}\u)\|_{\L^2(\Gamma_1)}= \|\partial J(\mathcal{L}\u,\mathcal{L}\u)\|_{\L^2(\Gamma_1)}&\leq C_1+(C_2+C_3)\|\mathcal{L}\| \|\u\|_{\V}  \\
	&\leq C_1+(C_2+C_3)\|\mathcal{L}\|s_1=s_2.
\end{align*}
Hence, $\Xi(\v_0,\w_0) \subset \mathcal{D} \ \text{ for all }\  (\v_0,\w_0) \in \mathcal{D}$ for positive constants $s_1, s_2$ in the definition \eqref{def-1}.
\vskip 0.2cm
\noindent
\textbf{Step II:} \emph{$\Xi$ has nonempty, closed and convex values.} To prove that $\Xi$ has closed values, consider a sequence $(\u_n, \boldsymbol{\chi}_n) \in \Xi(\v_0,\w_0)$ such that 
$$ \u_n \to \u , \quad \boldsymbol{\chi}_n \to \boldsymbol{\chi} \ \text{ as }\  n \to \infty.$$
We show that $(\u, \boldsymbol{\chi}) \in \Xi(\v_0,\w_0).$ Since $(\u_n,\boldsymbol{\chi}_n)\in\Xi(\v_0,\w_0),$ we actually have $\u_n=\Pi(\v_0,\w_0)$ for every $n$. Hence the sequence $\u_n$ is constant. Therefore
$\u_n=\u=\Pi(\v_0,\w_0)$. Moreover, from the definition of  $\Xi(\v_0,\w_0)$, we get 
$ \boldsymbol{\chi}_n \in  \partial J(\mathcal{L}\u,\mathcal{L}\u).$
 Since, by \cite[Proposition 3.23]{SMAOMS}, the Clarke subdifferential $\partial J(\w,\w)$ has closed values for all $\w \in \L^2(\Gamma_1)$. In particular, taking $\w = \mathcal{L}\u$, we conclude that
$$\partial J(\mathcal{L}\u,\mathcal{L}\u)$$ 
is a closed subset in $\L^2(\Gamma_1)$, which further implies $\boldsymbol{\chi} \in  \partial J(\mathcal{L}\u,\mathcal{L}\u).$ Hence, we get $(\u, \boldsymbol{\chi}) \in \Xi(\v_0,\w_0)$.

In order to prove  $\Xi(\v_0,\w_0)$ has convex values, consider $(\u_1,\boldsymbol{\chi}_1), (\u_2,\boldsymbol{\chi}_2) \in 
\Xi(\v_0,\w_0)$. Then from the definition of  $\Xi(\v_0,\w_0)$, we obtain 
$$ \u_1=\u_2=\u
 \ \text{ and } \ \boldsymbol{\chi}_1, \boldsymbol{\chi}_2 \in \partial J(\mathcal{L}\u,\mathcal{L}\u).$$

From \cite[Proposition 3.23]{SMAOMS}, it is well known that the Clarke subdifferential of a locally Lipschitz functional has convex values. Hence, for every $\lambda \in[0,1]$,
$$
\lambda \boldsymbol{\chi}_1 + (1-\lambda)\boldsymbol{\chi}_2
\in \partial J(\mathcal{L}\u,\mathcal{L}\u).$$
Therefore,
$$
\lambda (\u_1,\boldsymbol{\chi}_1)
+
(1-\lambda)(\u_2,\boldsymbol{\chi}_2)
=
\bigl(\u,
\lambda \boldsymbol{\chi}_1 + (1-\lambda)\boldsymbol{\chi}_2
\bigr)
\in \Xi(\v_0,\w_0).$$
This proves that $\Xi(\v_0,\w_0)$ is convex. Since
$(\v_0,\w_0) \in \mathcal{D}$ was arbitrary, the multifunction
$\Xi$ has convex values.
\vskip 0.2cm 
\noindent
\textbf{Step III:} \emph{The graph of $\Xi$ is sequentially weakly closed in $\mathcal{D} \times \mathcal{D}$.} For proving this, consider $({\v_0}_n, {\w_0}_n) \in \mathcal{D}, (\widehat{\v_0}_n, \widehat{\w_0}_n) \in \Xi({\v_0}_n, {\w_0}_n)$ with 
$$({\v_0}_n,{\w_0}_n) \xrightarrow{w} (\v_0, \w_0),~ (\widehat{\v_0}_n, \widehat{\w_0}_n) \xrightarrow{w} (\widehat{\v_0}, \widehat{\w_0}) \ \text{ in }\   \V\cap\L^{r+1} \times \L^2(\Gamma_1).$$ 
We have to show that $(\widehat{\v_0}, \widehat{\w_0}) \in \Xi(\v_0, \w_0)$. By the definition of $\Xi$, we have
\begin{equation}\label{2.1012}
\widehat{\v_0}_n = \Pi({\v_0}_n,{\w_0}_n) \ \text{ and } \ \widehat{\w_0}_n \in \partial J(\mathcal{L}\Pi({\v_0}_n, {\w_0}_n),\mathcal{L}\Pi({\v_0}_n, {\w_0}_n)).
	\end{equation}
By using the complete continuity of the map $\Pi$ proved in Lemma \ref{lemma-2}, we get \begin{align}\label{2.102}
\Pi({\v_0}_n,{\w_0}_n) \to \Pi(\v_0,\w_0) \ \text{ in }\  \V\cap\L^{r+1},
\end{align} 
and using continuity of the operator $\mathcal{L}$, we obtain
\begin{align}\label{2.103}
\mathcal{L}\Pi({\v_0}_n,{\w_0}_n) \to \mathcal{L}\Pi(\v_0,\w_0) \ \text{ in }\  \L^2(\Gamma_1).
\end{align} 
Thus, by combining \eqref{2.1012} with \eqref{2.102}, and using uniqueness of weak limits, it follows that
$$\widehat{\v_0} =\Pi(\v_0,\w_0),$$
and by using the closed graph of the generalized gradient $\partial J$ in $\L^2(\Gamma_1) \times (\L^2(\Gamma_1))_w$ (see \cite[Proposition 3.23]{SMAOMS}) along with \eqref{2.103}, we deduce 
$$\widehat{\w_0} \in \partial J(\mathcal{L}\Pi(\v_0,\w_0),\mathcal{L}\Pi(\v_0,\w_0)).$$
Hence, $(\widehat{\v_0}, \widehat{\w_0}) \in (\Pi(\v_0,\w_0), H(\mathcal{L}\Pi(\v_0,\w_0))) = \Xi(\v_0, \w_0)$, which proves the closedness of the graph of $\Xi$.

Hence, by applying Theorem \ref{theorem 2.10}, we deduce that there exists a fixed point of  $\Xi$ of \emph{Problem \ref{Problem 3.1}}, which means that $\u \in \V\cap\L^{r+1}$ such that $\u \in \mathbb{U}(\u)$ and there is $\w_0 \in \partial J(\mathcal{L}\u,\mathcal{L}\u)$, satisfying 
\begin{align*}
	\langle\mathcal{F}(\u)- \f,\v-\u\rangle+\langle \mathcal{L}^*\w_0,\v-\u \rangle+\Phi(\v)-\Phi(\u)  \geq 0 \ \text{ for all }\  \v \in \mathbb{U}(\u),
\end{align*}
which completes the proof.
\end{proof}

\begin{theorem}
Under  the assumptions of Theorem \ref{theorem-ex},   every solution $\u$ of Problem \ref{Problem 3.1} is the solution of Problem \ref{Problem 2.17}.
\end{theorem}
\begin{proof}
Since $\u \in \mathbb{U}(\u) \subset \V \cap \L^{r+1}$ is a solution of \emph{Problem \ref{Problem 3.1}} (see Theorem \ref{theorem-ex}), there exists $\w_0 \in \partial J(\mathcal{L}\u,\mathcal{L}\u)$ such that $\u$ satisfies the following variational inequality:
\begin{align}\label{2.1002}
	\langle\mathcal{F}(\u)- \f,\v-\u\rangle+\langle \w_0,\mathcal{L}\v-\mathcal{L}\u \rangle+\Phi(\v)-\Phi(\u)  \geq 0 \ \text{ for all }\  \v \in \mathbb{U}(\u).
\end{align}
Moreover, since $\w_0 \in \partial J(\mathcal{L}\u,\mathcal{L}\u)$,  the definition of the Clarke subdifferential (see \eqref{subgradient}) yield
\begin{align}\label{2.105}
{}_{\L^2(\Gamma_1)'}\langle \w_0, \boldsymbol{\zeta} \rangle_{\L^2(\Gamma_1)} \leq J^{0}(\mathcal{L}\u,\mathcal{L}\u;\boldsymbol{\zeta}) \ \text{ for all }\  \boldsymbol{\zeta} \in \L^2(\Gamma_1).
\end{align}
Thus, in particular, choosing $\boldsymbol{\zeta}= \mathcal{L}\v - \mathcal{L}\u$ in \eqref{2.105}, we get 
\begin{align}\label{2.1003}
\langle \w_0, \mathcal{L}\v - \mathcal{L}\u \rangle \leq J^{0}(\mathcal{L}\u,\mathcal{L}\u;\mathcal{L}\v - \mathcal{L}\u) \ \text{ for all }\  \v \in \mathbb{U}(\u).
\end{align}
Hence, by substituting \eqref{2.1003} in \eqref{2.1002}, we obtain
\begin{align*}
\langle\mathcal{F}(\u),\v-\u\rangle+J^{0}(\mathcal{L}\u,\mathcal{L}\u;\mathcal{L}\v - \mathcal{L}\u) +\Phi(\v)-\Phi(\u)  \geq \langle \f,\v-\u\rangle \ \text{ for all }\  \v \in \mathbb{U}(\u),
\end{align*}
which shows that $\u \in \mathbb{U}(\u) \subset \V\cap\L^{r+1}$ is the solution of  \emph{Problem \ref{Problem 2.17}}.
\end{proof}

\section{Conclusion}\label{sec5}\setcounter{equation}{0}
In this paper, we have investigated a quasi-variational--hemivariational inequality arising from the convective Brinkman--Forchheimer extended Darcy (CBFeD) model describing the flow of an incompressible Bingham fluid under mixed boundary conditions with nonmonotone friction. The main contribution of this work is the establishment of the existence of weak solutions to the proposed problem. By deriving a suitable weak formulation, we reformulated the governing CBFeD system as a Bingham-type quasi-variational--hemivariational inequality. The resulting formulation incorporates several challenging features, including the nonlinear convection term, Brinkman--Forchheimer damping effects, a convex potential, non-smooth and nonmonotone slip boundary conditions, and solution-dependent unilateral constraints.

The analysis is based on the Kakutani--Ky Fan fixed-point theorem for set-valued mappings. Employing this framework, we first established the existence of weak solutions to the associated quasi-variational inequality. We then proved that every weak solution of the quasi-variational inequality also satisfies the corresponding quasi-variational--hemivariational inequality, thereby demonstrating the well-posedness and consistency of the proposed variational formulation.

The results presented in this paper provide a rigorous mathematical framework for the analysis of CBFeD models for incompressible Bingham fluids with state-dependent constraints and non-smooth boundary interactions. Several interesting problems remain open for future investigation. These include the study of uniqueness and stability of weak solutions, the development and convergence analysis of numerical approximation schemes, the extension of the present framework to time-dependent CBFeD models, and the investigation of optimal control, inverse problems, and coupled multiphysics systems involving non-Newtonian fluids.

\begin{appendix}{\Alph{section}}

	%	\numberwithin{equation}{section}
	\section{}\label{Appendix}\setcounter{equation}{0}
		The appendix  establishes a fundamental integration-by-parts identity for the trilinear form associated with the convective term and the symmetric gradient operator $\mathbb{D}$. We first derive the identity for smooth functions $\mathrm{C}^{\infty}(\overline{\Omega};\mathbb{R}^d)$ and then extend the result to the Sobolev space $\mathrm{H}^1(\Omega;\mathbb{R}^d)$ via a standard density argument.
For $\u, \v \in \mathrm{C}^{\infty}(\overline{\Omega};\mathbb{R}^d)$, by using integration by parts,  we have
\begin{align}\label{1.30}
	\int_{\Omega} (\u \otimes \u) : \mathbb{D}\v \d x \nonumber
	&= \frac{1}{2} \sum_{i,j=1}^{d} \int_{\Omega}
	\left(
	y_i(x) y_j(x) \frac{\partial z_i(x)}{\partial x_j}
	+ y_i(x) y_j(x)\frac{\partial z_j(x)}{\partial x_i}
	\right) \d x \nonumber \\ 
	&= \frac{1}{2} \sum_{i,j=1}^{d}
	\Bigg(
	- \int_{\Omega} z_i(x) \frac{\partial}{\partial x_j} (y_i(x) y_j(x)) \d x
	+ \int_{\Gamma} y_i(x) y_j(x) z_i(x) \nu_j \d\Gamma 
	\nonumber \\
	&\qquad
	- \int_{\Omega} z_j(x) \frac{\partial}{\partial x_i}(y_i(x) y_j(x)) \d x
	+ \int_{\Gamma} y_i(x) y_j(x) z_j(x) \nu_i \d\Gamma
	\Bigg)\nonumber 	\\
	&= -\frac{1}{2} \sum_{i,j=1}^{d} \int_{\Omega}
	\left(z_i(x) \frac{\partial y_i(x)}{\partial x_j} y_j(x)
	+ z_j(x) \frac{\partial y_j(x)}{\partial x_i} y_i(x)\right) \d x  \nonumber \\
	&\qquad	-\frac{1}{2} \sum_{i,j=1}^{d} \int_{\Omega}
	\left( y_i(x) z_i(x) \frac{\partial y_j(x)}{\partial x_j}
	+ y_j(x) z_j(x) \frac{\partial y_i(x)}{\partial x_i}
	\right) \d x \nonumber \\
	&\qquad
	+  \frac{1}{2} \sum_{i,j=1}^{d} \int_{\Gamma}
	\bigg(
	(y_i(x) z_i(x))(y_j(x) \nu_j)
	+ (y_j(x) z_j(x))(y_i(x) \nu_i)\bigg) \d\Gamma \nonumber \\
	&= - \frac{1}{2} \Bigg(
	b(\u,\u,\v) + b(\u,\u,\v)
	+ \int_{\Omega} \bigl( (\u\cdot \v)(\nabla\cdot \u) + (\u\cdot \v)(\nabla\cdot \u) \bigr) \d x
	\Bigg) \nonumber \\
	&\qquad
	+  \frac{1}{2} \int_{\Gamma}
	\bigg((\u\cdot \v) y_\nu + (\u\cdot \v) y_\nu \bigg)\d\Gamma \nonumber \\
	&= - b(\u,\u,\v)
	- \int_{\Omega} (\u\cdot \v)(\nabla\cdot\u) \d x
	+ \int_{\Gamma} (\u\cdot \v) y_\nu \d\Gamma.
\end{align}
Since $\mathrm{C}^{\infty}(\overline{\Omega};\mathbb{R}^d)$ is dense in $\mathrm{H}^1(\Omega;\mathbb{R}^d)$, the above equality \eqref{1.30} holds for all $\u,\v \in \mathrm{H}^1(\Omega;\mathbb{R}^d)$.

\end{appendix}

\medskip\noindent
\textbf{Acknowledgments:} The first author gratefully acknowledges the financial support received from the Ministry of Education, Government of India, under the Prime Minister Research Fellowship (PMRF ID: 2803609). M. T. Mohan acknowledges the support of the National Board for Higher Mathematics (NBHM), Department of Atomic Energy, Government of India, through Project No. 02011/13/2025/NBHM(R.P)/R\&D II/1137.

\medskip\noindent	{\bf  Declarations:} 

\noindent 	{\bf  Ethical Approval:}   Not applicable 

	\noindent  {\bf   Author Contributions: } All authors contributed equally.
	 
\noindent  {\bf   Competing interests: } The author declare no competing interests. 

\noindent 	{\bf   Availability of data and materials: } Not applicable.

\end{document}